\documentclass[a4paper,10pt,abstract=on]{scrartcl}

\usepackage[pagewise]{lineno}\nolinenumbers

\usepackage{sidecap}
\usepackage[utf8]{inputenc}
\usepackage[T1]{fontenc}
\usepackage[english]{babel}
\usepackage{hyphenat}
\usepackage{amsmath,amssymb,amsfonts}
\usepackage{amsthm}
\usepackage[svgnames]{xcolor}
\usepackage{mathtools}
\mathtoolsset{showonlyrefs}
\usepackage{geometry}
\usepackage{bm}
\usepackage[shortlabels]{enumitem}
\usepackage{csquotes}
\usepackage{booktabs}
\usepackage{cite}
\usepackage{subcaption}
\usepackage[hidelinks]{hyperref}
\usepackage{multirow}
\usepackage{multicol}

\usepackage[svgnames]{xcolor}
\usepackage{mathtools}
\usepackage{bm}
\usepackage{csquotes}
\usepackage{booktabs}
\mathtoolsset{showonlyrefs}
\usepackage[shortlabels]{enumitem}
\usepackage[hidelinks]{hyperref}
\usepackage{comment}

\theoremstyle{plain}
\newtheorem{theorem}{Theorem} 
\newtheorem{lemma}[theorem]{Lemma}
\newtheorem{proposition}[theorem]{Proposition}
\newtheorem{corollary}[theorem]{Corollary}

\theoremstyle{definition}

\newtheorem{algorithm}[theorem]{Algorithm}

\theoremstyle{remark}
\newtheorem{remark}[theorem]{Remark}

\providecommand{\keywords}[1]{
  \small	
  \textbf{\textit{Keywords---}} #1}

\providecommand{\msc}[1]{
  \small	
  \textbf{\textit{2020 AMS Mathematics Subject Classification ---}} #1}

\DeclareMathOperator*{\argmin}{arg\,min}

\DeclareMathOperator{\proj}{proj}

\DeclareMathOperator{\rank}{rk}

\newcommand{\I}{\mathrm i}
\newcommand{\J}{\mathrm j}
\newcommand{\K}{\mathrm k}

\newcommand{\R}{\mathbb R}
\newcommand{\C}{\mathbb C}
\newcommand{\HH}{\mathbb H}

\newcommand{\Fun}[1]{\mathcal{#1}}

\newcommand{\FJ}{\Fun J}
\newcommand{\FK}{\Fun K}

\newcommand{\FG}{\Fun G}
\newcommand{\FQ}{\Fun Q}

\newcommand{\Mat}[1]{\bm{#1}}
\newcommand{\MA}{\Mat A}
\newcommand{\MB}{\Mat B}
\newcommand{\MC}{\Mat C}
\newcommand{\MI}{\Mat I}

\newcommand{\MSigma}{\Mat \Sigma}

\newcommand{\MM}{\Mat M}
\newcommand{\MP}{\Mat P}
\newcommand{\MQ}{\Mat Q}
\newcommand{\MS}{\Mat S}
\newcommand{\MR}{\Mat R}
\newcommand{\MU}{\Mat U}
\newcommand{\MV}{\Mat V}
\newcommand{\MW}{\Mat W}
\newcommand{\MX}{\Mat X}
\newcommand{\MY}{\Mat Y}
\newcommand{\MZ}{\Mat Z}
\newcommand{\Mzero}{\Mat 0}

\newcommand{\Mcomplex}{\MM}
\newcommand{\Mquater}{\MM}

\newcommand{\Vek}[1]{\bm{#1}}

\newcommand{\Vx}{\Vek x}
\newcommand{\Vy}{\Vek y}
\newcommand{\Vz}{\Vek z}

\newcommand{\Vw}{\Vek w}
\newcommand{\Vv}{\Vek v}
\newcommand{\Vr}{\Vek r}
\newcommand{\Vq}{\Vek q}
\newcommand{\Vell}{\Vek \ell}

\newcommand{\sphere}{\mathbb S}
\newcommand{\Csphere}{\sphere_{\C}}
\newcommand{\Hsphere}{\sphere_{\HH}}
\newcommand{\SO}{{\mathrm{SO}}}

\newcommand{\cC}{\mathcal C}

\newcommand{\tT}{\mathrm{T}}
\newcommand{\tF}{\mathrm{F}}
\newcommand{\tH}{\mathrm{H}}

\begin{document}

\title{Denoising of Sphere- and SO(3)-Valued Data 
by Relaxed Tikhonov Regularization}

\author{Robert Beinert\thanks{R. Beinert is with the Institute of Mathematics,
	Technische Universit\"at Berlin, Stra\ss{}e des 17. Juni 136,
        10623 Berlin, Germany.}
        \and
    Jonas Bresch\thanks{J. Bresch is with the Institute of Mathematics,
	Technische Universit\"at Berlin, Stra\ss{}e des 17. Juni 136,
        10623 Berlin, Germany.}
        \and 
    Gabriele Steidl\thanks{G. Steidl is with the Institute of Mathematics,
	Technische Universit\"at Berlin, Stra\ss{}e des 17. Juni 136,
        10623 Berlin, Germany.}%
}

\maketitle

\begin{abstract}

Manifold-valued signal- and image processing has
received attention due to modern image acquisition techniques.
Recently,
a convex relaxation 
of the otherwise nonconvex Tikhonov-regularization
for denoising circle-valued data 
has been proposed by Condat (2022).
The circle constraints are here encoded 
in a series of low-dimensional, positive semi-definite matrices.
Using Schur complement arguments,
we show 
that the resulting variational model can be simplified
while leading to the same solution.
The simplified model can be generalized to higher
dimensional spheres and to $\SO(3)$-valued data, 
where we rely on the quaternion representation of the latter.
Standard algorithms from convex analysis can be applied to solve
the resulting convex minimization problem.
As proof-of-the-concept, we use the alternating direction method of multipliers 
to demonstrate the denoising behavior of the proposed method.
In a series of experiments, 
we demonstrate the numerical convergence 
of the signal- or image values to the underlying manifold.
\end{abstract}


\hspace{0.25cm}

\msc{94A08, 94A12, 65J22, 90C22, 90C25}

\section{Introduction}

With the emergence of modern acquisition techniques 
which produce manifold-valued signals and images, 
research in this field focuses new challenges.  
Circle-valued data appear in interferometric synthetic aperture radar \cite{BRFa2000,DDT2011,RHJL2000}, 
color image restoration in HSV or LCh spaces \cite{NikSte14}, 
Magnetic Resonance Imaging \cite{LEHS2008}, 
biology \cite{SRL2005}, psychology \cite{CK2018}
and various other applications involving
the phase of Fourier transformed data for instance. 
Signals and images with values on the 2-sphere
play a role when dealing with 3d directional information \cite{ASWK1993,KS2002,VO2002} 
or in the processing of color images 
in the chromaticity-brightness setting \cite{BB2018,PPS2017,QKL2010}. 
The rotation group $\SO(3)$ was considered in tracking and motion analysis 
and in the analysis of electron back-scatter diffraction data (EBSD) \cite{BHS11,BHJPSW10,GNHSLP2022}.
Quaternion-valued data also arises 
in color image restoration,
concerning the hue, saturation, and value color space \cite{ZMWW19}.

The denoising of circle- and sphere-valued data was addressed by several methods
as lifting procedures \cite{CS13,SC11,LSKC13}, 
variational approaches with total variation-like (TV-like) regularizers \cite{BBSW16,WDS14} 
and half-quadratic minimization models
like the iteratively re-weighted least squares method \cite{BerChaHiePerSte16,GS14}.
These methods either enlarge the dimension of the problem drastically,
especially for lifting methods,
have restricted convergence guarantees, 
or rely on data on Hadamard manifolds, which spheres are not.
There are few stochastic approaches in manifold-valued data processing 
as the work of Pennec's group, see for instance \cite{pennec2006,LNPS2017}, 
where a patch-based approach
via a minimum mean square estimation (MMSE) model 
on the respective manifold is applied. 
For an overview, we also refer to \cite{BerLauPerSte19}.
For denoising matrix-manifolds and matrix-valued images,
which become a mayor part of computer vision, image processing and motion analysis,
the rotation group SO($n$),
the special-euclidean group SE($n$)
and the symmetric positive definite matrices SPD($n$)
are of mayor interest.
For instance, 
\cite{RTKB14,RBBTK12} consider a TV regularizer on the different Lie-groups
for rigid 3D motions,
minimizing the irregularity of the motion field.

In this paper, we focus exclusively on 
the important simple variational model with quadratic Tikhonov regularization
that penalizes the squared differences of
adjacent values on a graph, 
which promotes the property that such values are close
to each other. 
Unfortunately, due to the manifold constrained, 
this is still a difficult nonconvex problem.
Recalling the convexity relaxed model of Condat \cite{condat_1D2D} 
for circle-valued data,
we show how this relaxed model can be nicely simplified 
without losing any information.
Based on Schur complement arguments, 
we will see that our model leads to the same minimizers
while reducing the number of parameters 
and consequently making the computation more efficient.
In a preprint \cite{condat_3D}, which came without any numerical results, 
Condat suggested to generalize his model to data on the 2-sphere.
In this paper, we will see that this approach can be simplified as well.
Moreover, it just appears as a special case of quaternion-valued data processing.
The latter one also leads to a relaxed convex model for $\SO(3)$-valued data.
In the proceeding \cite{BeiBre2024},
we introduce a similar approach 
for anisotropic TV. 
Since the TV functional has another structure 
than the quadratic Tikhonov regularizer,
we end up with a completely different convexification.

Outline of the paper:
in Section~\ref{sec:tikh}, we reconsider Condat's relaxation 
of the nonconvex Tikhonov model for denoising circle-valued data.
We address the model from two different points of view depending 
whether we embed 
the circle into the complex numbers $\C$ 
or into the real numbers $\R^2$.
Both viewpoints make it easier to understand the later extension 
of the approach to the quaternion setting.
Then, we introduce our simplified model and prove 
that this model has the same solutions as the original relaxed formulation
by using Schur complement arguments.
In Section~\ref{sec:sd}, we will see that our simplified model 
can be generalized to higher dimensional spheres
in a straightforward way.
Next, in Section~\ref{sec:sdc}, we explain how the nonconvex Tikhonov
model as well as its relaxation 
and simplification can be generalized to $\SO(3)$-valued data.
The approach relies on the parameterization of $\SO(3)$ by quaternions.
Remarkably, it also includes 2-sphere-valued data.
We describe the ADMM for finding a minimizer of our relaxed models in Section~\ref{sec:algs}. 
Finally, Section~\ref{sec:numerics}
provides some proof-of-the-concept denoising examples 
for \mbox{circle-,} sphere- and $\SO(3)$-valued data.
Although we cannot prove the tightness of our relaxation,
we observe that the calculated numerical solutions
solve the original nonconvex problem,
i.e.\ the denoised signals are again 
\mbox{circle-,} sphere- or $\SO(3)$-valued.
In case of circle-valued data,
we compare our simplified approach
with the original relaxation in \cite{condat_1D2D},
where we observe significant improvements 
concerning the computational complexity
and numerical convergence.
Additionally,
we compare our approach with the TV denoiser from \cite{BerLauSteWei14},
which is based on geodesic distances,
where we observe significant improvements 
concerning the computation time and restoration error.

\section{Denoising of Circle-Valued Data}   \label{sec:tikh}

Let $G = (V, E)$ be a connected, undirected graph,
where $V \coloneqq \{1, \dots, N \}$ denotes the set of vertices,
and $E \subset \{ (n,m) \in V \times V : n < m\}$ the set of edges.
The cardinality of $E$ is henceforth denoted by $M \coloneqq \lvert E \rvert$.
Our aim is to denoise a disturbed circle-valued signal on $G$.
The circle can be either embedded 
into the complex numbers $\C$ 
or into the two-dimensional Euclidean space $\R^2$.
The former approach has been especially considered by Condat \cite{condat_1D2D}
and is summarized in Section~\ref{sec:C_val_mod}.
The latter approach proposed by us 
is studied in Section~\ref{sec:R2_simp_val_mod} and
can be immediately generalized
to tackle sphere-valued and $\SO(3)$-valued data,
see Section~\ref{sec:sd} and \ref{sec:sdc}.
To show that the complex- and real-valued model
are in fact equivalent,
we introduce an in-between model in Section~\ref{sec:R2_val_mod}
interpreting the complex-valued approach
itself as real-valued problem.

\subsection{$\C$-Valued Model}
\label{sec:C_val_mod}

Let $\Csphere \coloneqq \{ z \in \C : \lvert z \rvert = 1\}$
denote the (complex-valued) unit circle.
Our aim is to recover an $\Csphere$-valued signal 
$x \coloneqq (x_n)_{n \in V} \in \Csphere^{N}$ on $G$
from noisy measurements 
$y \coloneqq (y_n)_{n \in V} \in {\Csphere}^N$. Note that the following considerations work also for $y \in {\C}^N$.
A straightforward approach would be to search for the minimizer of the 
Tikhonov-like regularized functional
\begin{linenomath*}
\begin{equation}     \label{eq:class-tik}
    \argmin_{x \in \Csphere^N} 
    \sum_{n \in V} 
    \frac{w_n}{2} \, \lvert x_n - y_n \rvert^2
    + \smashoperator{\sum_{(n,m) \in E}} 
    \frac{\lambda_{(n,m)}}{2} \, \lvert x_n - x_m \rvert^2, 
\end{equation}
\end{linenomath*}
where $w \coloneqq (w_n)_{n \in V} \in \R^N_+$ 
and $\lambda \coloneqq (\lambda_{(n,m)})_{(n,m)\in E} \in \R^M_+$
are positive weights.
Unfortunately, due to the constraints $x_n \in \Csphere$ , 
this problem is nonconvex.
In the first step,
to derive a convex relaxation,
we exploit $\lvert x_n \rvert = 1$

and hence we derive
\begin{linenomath*}
\begin{equation*}
    \lvert x_n - y_n\rvert^2 
    = \lvert x_n\rvert^2  - 2\Re[x_n \bar y_n] + \lvert y_n\rvert^2
    = - 2\Re[x_n \bar y_n] + \text{const}
\end{equation*}
\end{linenomath*}
and similarly 
\begin{linenomath*}
\begin{equation*}
    \lvert x_n - x_m\rvert^2 
    = \lvert x_m\rvert^2  - 2\Re[\bar x_m x_n]+ \lvert y_n\rvert^2
    = - 2\Re[\bar x_m x_n] + \text{const},
\end{equation*}
\end{linenomath*}
where $\Re[z] = \Re[\alpha + \I\beta] = \alpha$ 
for $z = \alpha + \I\beta \in \C$ with $\alpha, \beta \in \R$.
We rewrite
the original problem \eqref{eq:class-tik} as
\begin{linenomath*}
\begin{equation}    
    \label{eq:inner-tik}
    \argmin_{x \in \Csphere^N} \,
    - \sum_{n \in V} 
    w_n \, \Re[x_n \bar y_n] \,
    - \smashoperator{\sum_{(n,m) \in E}}
    \lambda_{(n,m)} \, \Re[\bar x_m x_n].
\end{equation}
\end{linenomath*}
Introducing 
$r \coloneqq (r_{(n,m)})_{(n,m) \in E} \in \C^M$,
we rearrange the minimization problem \eqref{eq:inner-tik} into
\begin{linenomath*}
\begin{equation}     
    \label{eq:ext-tik}
    \argmin_{x \in \Csphere^N , r \in \C^M} 
    \!\!\!\FJ(x,r)
    \text{ s.t. }
    r_{(n,m)} = \bar x_m x_n 
\end{equation}
\end{linenomath*}
for all $(n,m) \in E$ with
\begin{linenomath*}
\begin{equation}
    \FJ(x,r) 
    \coloneqq
    - \sum_{n \in V}
    w_n \, \Re[ x_n \bar y_n] 
    - \smashoperator{\sum_{(n,m) \in E}} 
    \lambda_{(n,m)} \, \Re[r_{(n,m)}].
\end{equation}
\end{linenomath*}
The core idea of Condat's convex relaxation \cite{condat_1D2D}
is to encode the nonconvex constraint $x \in \Csphere^N$
into a positive semi-definite, rank-one matrix.

\begin{lemma}
    \label{lem:1}
    Let $n,m \in V$ and $(n,m) \in E$.
    Then $x_n,x_m \in \Csphere$
    and $r_{(n,m)} = \bar x_m x_n $ 
    if and only if
    \begin{linenomath*}
    \begin{equation}    
        \label{eq:mat_s1}   
        P_{(n,m)}
        \coloneqq
        \begin{bmatrix}
            1 &  x_n & x_m \\
            \bar x_n & 1 & \bar r_{(n,m)} \\
            \bar x_m & r_{(n,m)} & 1
        \end{bmatrix}
        \in \C^{3 \times 3}
    \end{equation}
    \end{linenomath*}
    is positive semi-definite and has rank one.
\end{lemma}

\begin{proof}
    Let $x_n,x_m \in \Csphere$ and $r_{(n,m)} = \bar x_m x_n $,
    and denote the conjugation and transposition by $\cdot^\tH$.
    Then
    $
    P_{(n,m)} = (1, x_n, x_m)^\tH (1, x_n, x_m)
    $
    has rank one and is positive semi-definite. 
    Conversely, 
    positive semi-definitness and the rank-one condition imply that
    $P_{(n,m)}$ has the form
    \begin{linenomath*}
    \begin{equation*}
        P_{(n,m)} = (a,b,c)^\tH (a, b, c) = 
        \begin{bmatrix}                
            \bar a a &  \bar a b & \bar a c \\
            \bar b a &\bar b b & \bar b a \\
            \bar c a & \bar c b & \bar c c
        \end{bmatrix}
    \end{equation*}
    \end{linenomath*}
    for some $a,b,c \in \C$.
    Comparison with \eqref{eq:mat_s1} yields $|a|= |b| = |c| = 1$
    and thus $|x_n| = |\bar a| |b| = 1$, $|x_m| = |\bar a| |c| = 1$ 
    as well as 
    $r_{(n,m)} = \bar c b = \bar c a \bar a b = \bar x_m x_n$.
\end{proof}

We denote the rank function by $\rank$.
Based on Lemma~\ref{lem:1},
Problem \eqref{eq:ext-tik} and thus the original formulation \eqref{eq:class-tik} become
\begin{linenomath*}
\begin{equation}     \label{eq:ext-tik_end}
    \argmin_{x \in \C^N , r \in \C^M} 
    \FJ(x,r)
    \text{ s.t. }
    P_{(n,m)} \succeq 0  \; \text{and} \; \rank(P_{(n,m)}) = 1 
\end{equation}
\end{linenomath*}
for all $(n,m) \in E$.
Neglecting the rank-one constraint,
Condat \cite{condat_1D2D} proposes to solve the relaxation:
\begin{linenomath*}
\begin{align}     \label{eq:conv-tik}
    &\textbf{relaxed complex model}
    \notag \\
    &\argmin_{x \in \C^N, r \in \C^M} 
    \FJ (x, r) 
    \quad \text{s.t.} \quad 
    P_{(n,m)} \succeq 0
\end{align}
\end{linenomath*}
for all $(n,m) \in E$.
Since \eqref{eq:conv-tik} is a convex optimization problem,
we may apply standard methods from convex analysis
to obtain numerical solutions.

\subsection{Related $\R^2$-Valued Model}
\label{sec:R2_val_mod}

Alternatively, 
the sphere can be embedded in the Euclidean space $\R^2$
equipped with the Euclidean norm $\lVert \cdot \rVert$
and inner product $\langle \cdot, \cdot \rangle$.
To emphasize the relation to the complex-valued model,
we identify the complex number $z \in \C$ 
with the two-dimensional Euclidean vector $\Vz \coloneqq (\Re[z], \Im[z])^\tT \in \R^2$,
where~$\cdot^\tT$ denotes the transposition.
Against this background,
we denote the first and second component of $\Vz$ by 
$\Re[\Vz] \coloneqq \Vz_1$
and $\Im[\Vz] \coloneqq \Vz_2$
respectively.
Defining
$\sphere_1 \coloneqq \{\Vz \in \R^2: \lVert\Vz\rVert = 1\}$,
we now want to recover an $\sphere_1$-valued signal 
$\Vx \coloneqq (\Vx_n)_{n \in V} \in \sphere_1^N$ on $G$
from noisy measurements $\Vy \coloneqq (\Vy_n)_{n \in V} \in (\mathbb S_1)^N$.
Exploiting $\lvert z \rvert = \lVert \Vz \rVert$ for all $z \in \C$,
we rewrite \eqref{eq:class-tik} into the real-valued model
\begin{linenomath*}
\begin{equation}     \label{eq:class-tik_real_1}
    \argmin_{\Vx \in \sphere_1^N} 
    \sum_{n \in V} 
    \frac{w_n}{2} \, \lVert \Vx_n - \Vy_n \rVert^2
    + \smashoperator{\sum_{(n,m) \in E}} 
    \frac{\lambda_{(n,m)}}{2} \, \lVert \Vx_n - \Vx_m \rVert^2.
\end{equation}
\end{linenomath*}
To transfer the relaxation of the complex model,
we exploit that the complex multiplication 
can be realized using the matrix representation of $z \in \C$
given by
\begin{linenomath*}
\begin{equation}    \label{eq:2x2_real_pres}
    \Mcomplex(z) 
    \coloneqq
    \Mcomplex(\Vz)
    \coloneqq
    \left[
    \begin{smallmatrix}
        \Re[\Vz] & -\Im[\Vz] \\ 
        \Im[\Vz] & \Re[\Vz]
    \end{smallmatrix}
    \right].
\end{equation}
\end{linenomath*}
Introducing the variables $\Vr_{(n,m)} = \Mcomplex(\Vx_m)^\tT \Vx_n \in \R^2$,
the real-valued version of \eqref{eq:ext-tik} reads as
\begin{linenomath*}
\begin{equation}     
    \label{eq:ext-tik_real}
    \argmin_{\Vx \in \sphere_1^N , \Vr \in (\R^2)^M} 
    \FJ(\Vx, \Vr)
    \quad \text{s.t.} \quad
    \Vr_{(n,m)}	= \Mcomplex(\Vx_m)^\tT \Vx_n
\end{equation}
\end{linenomath*}
for all $(n,m) \in E$ with
\begin{linenomath*}
\begin{equation}\label{7'}
    \FJ(\Vx, \Vr) 
    \coloneqq
    - \sum_{n \in V}
    w_n \, \langle \Vx_n , {\Vy}_n \rangle 
    - \smashoperator{\sum_{(n,m) \in E}} 
    \lambda_{(n,m)} \, \Re[\Vr_{(n,m)}].
\end{equation}
\end{linenomath*}
Note that $\FJ$ does not depend on $\Im[\Vr_{(n,m)}]$
and is an artifact of the complex viewpoint.
In Section~\ref{sec:R2_simp_val_mod}, 
we show that this auxiliary variable can indeed be dropped 
from the optimization without changing the minimizer.
Denoting the $d \times d$ identity matrix by $\MI_d$,
where the index is omitted if the dimension is clear,
we analogously encode the nonconvex constraint of \eqref{eq:ext-tik_real}
into a matrix expression.

\begin{lemma} 
    \label{lem:2}
    Let $n,m \in V$ and $(n,m) \in E$.
    Then $\Vx_n,\Vx_m \in \sphere_1$
    and $\Vr_{(n,m)} = \Mcomplex (\Vx_m)^\tT \Vx_n$
    if and only if the block matrix
    \begin{linenomath*}
    \begin{equation}    \label{eq:mat_r2}
      \MP_{(n,m)} 
      \coloneqq
      \left[
        \begin{smallmatrix}
          \MI_2 & \Mcomplex(\Vx_n) & \Mcomplex(\Vx_m)\\
          \Mcomplex(\Vx_n)^\tT & \MI_2 &   \Mcomplex(\Vr_{(n,m)})^\tT\\
          \Mcomplex(\Vx_m)^\tT & \Mcomplex(\Vr_{(n,m)}) &  \MI_2
        \end{smallmatrix}
      \right]
      \in \R^{6 \times 6}
    \end{equation}
    \end{linenomath*}
    is positive semi-definite and has rank two.
\end{lemma}

\begin{proof}
Let the conditions for $\Vx_n, \Vx_m$ and $\Vr_{(n,m)}$ be fulfilled.
Since  $\Mcomplex(\Vx_n)^\tT \Mcomplex(\Vx_n) = \lVert \Vx_n\rVert^2 \MI_2$ and
$\Mcomplex(\Vx_m)^\tT  \Mcomplex(\Vx_n) = \Mcomplex(\Vr_{(n,m)})$, we obtain
\begin{linenomath*}
\begin{equation*}
    \MP_{(n,m)} = 
    [\MI_2, \Mcomplex(\Vx_n), \Mcomplex(\Vx_m)]^\tT
    [\MI_2, \Mcomplex(\Vx_n), \Mcomplex(\Vx_m)]
\end{equation*}
\end{linenomath*}
implying $\rank(\MP_{(n,m)}) = 2$ and $\MP_{(n,m)} \succeq 0$.
The opposite direction follows similarly as in the proof of Lemma \ref{lem:1}.
\end{proof}

Due to Lemma~\ref{lem:2},
the real-valued formulation \eqref{eq:ext-tik_real} becomes
\begin{linenomath*}
\begin{equation}     \label{eq:ext-tik_end_real}
    \smashoperator{\argmin_{\substack{\Vx \in (\R^2)^N \\ \Vr \in (\R^2)^M}}}
    \FJ(\Vx, \Vr)
    \text{ s.t. }
    \MP_{(n,m)} \succeq 0 \text{ and } \rank(\MP_{(n,m)}) = 2 
\end{equation}
\end{linenomath*}
for all $(n,m) \in E$.
Since obviously $\MP_{(n,m)} \succeq 0$ if and only if $P_{(n,m)} \succeq 0$
from \eqref{eq:mat_s1},
neglecting the rank-two constraint in \eqref{eq:ext-tik_end_real}
yields the relaxed formulation:
\begin{linenomath*}
\begin{align}     \label{eq:conv-tik_real}
    &\textbf{relaxed real model}
    \notag \\
    &\argmin_{\Vx \in (\R^2)^N, \Vr \in (\R^2)^M }
    \FJ(\Vx, \Vr)
    \quad \text{s.t.} \quad
    \MP_{(n,m)} \succeq 0  
\end{align}
\end{linenomath*}
for all $(n,m) \in E$,
which is just the real-valued version of 
complex relaxation \eqref{eq:conv-tik}. 

\subsection{Simplified $\R^2$-Valued Model}
\label{sec:R2_simp_val_mod}

Note that the seconds variables 
$\Vr_{(n,m)} = \Mcomplex(\Vx_m)^\tT \Vx_n$ 
in \eqref{eq:ext-tik_real}
are artificial and originate form the complex-valued model.
Therefore the minimization problem \eqref{eq:ext-tik_real}
seems to be needlessly puffed-up
since we also optimize with respect to the variables $\Im(\Vr_{(n,m)})$
that do not influence the objective 
$\mathcal J$ in \eqref{7'} at all.
To this end,  
we propose to reformulate the real-valued problem \eqref{eq:class-tik_real_1} as
\begin{linenomath*}
\begin{equation}
    \label{eq:ext-real-tik}
    \argmin_{\Vx \in \sphere_1^N , \Vell \in \R^M} 
    \FK(\Vx,\Vell)
    \quad \text{s.t.} \quad
    \Vell_{(n,m)} = \langle \Vx_n, \Vx_m \rangle
\end{equation}
\end{linenomath*}
for all $(n,m) \in E$ with 
\begin{linenomath*}
\begin{equation}
    \FK(\Vx,\Vell) 
    \coloneqq
    - \sum_{n \in V}
    w_n \, \langle \Vx_n, \Vy_n\rangle 
    - \smashoperator{\sum_{(n,m) \in E}} 
    \lambda_{(n,m)} \, \Vell_{(n,m)},
\end{equation}
\end{linenomath*}

using a similar rewriting process for the squared 2-norm as in the previous sections.
As before and with a similar proof,
the nonconvex constraints $\Vx_n \in \sphere_1$
and $\Vell_{(n,m)} = \langle \Vx_n, \Vx_m \rangle$ may be encoded using 
the positive semi-definiteness and a rank condition of an appropriate matrix.

\begin{lemma} \label{lem:3}
    Let $n,m \in V$ and $(n,m) \in E$.
    Then $\Vx_n, \Vx_m \in \sphere_1$ 
    and $\Vell_{(n,m)} = \langle \Vx_n, \Vx_m \rangle$
    if and only if
    the block matrix
    \begin{linenomath*}
    \begin{equation}    \label{eq:mat_real}
        \MQ_{(n,m)}
        \coloneqq
        \begin{bmatrix}
            \MI_2 & \Vx_n & \Vx_m \\
            \Vx_n^\tT & 1 & \Vell_{(n,m)} \\
            \Vx_m^\tT & \Vell_{(n,m)} & 1
        \end{bmatrix}
        \in \R^{4 \times 4}
    \end{equation}
    \end{linenomath*}
    is positive semi-definite and has rank two. 
\end{lemma}

In this paper, we relax the rank-two constraint and
propose to solve our new relaxed formulation:
\begin{linenomath*}
\begin{align}    
&\textbf{simplified relaxed real model}\\[1ex]
&\argmin_{\Vx \in (\R^2)^N, \Vell \in \R^M} 
    \FK (\Vx, \Vell) 
    \quad \text{s.t.} \quad 
    \MQ_{(n,m)} \succeq 0 \label{eq:conv-real-tik:2} 
\end{align}
\end{linenomath*}
for all $(n,m) \in E$.
Our convex model \eqref{eq:conv-real-tik:2} is simpler than \eqref{eq:conv-tik_real}, since
$\Vr_{(n,m)} \in \R^2$ is replaced by $\Vell_{(n,m)} \in \R$.
Nevertheless, both problems lead to the same
solutions.
To show this, we need a result on the Schur complement of matrices.

\begin{proposition}[{\!\!\cite[p. 495]{HJ13}}]     \label{prop:schur-comp}
    For invertible $\MA \in \R^{\ell\times\ell}$, it holds
    \begin{linenomath*}
    $$
    \MW = \left[
        \begin{smallmatrix}
            \MA & \MC \\
            \MC^\tT & \MB
        \end{smallmatrix}
        \right]
       =
       \left[
        \begin{smallmatrix}
\MI & \Mzero \\
            \MC^\tT \MA^{-1}& \MI
        \end{smallmatrix}
        \right] 
        \left[
        \begin{smallmatrix}
\MA & \Mzero \\
\Mzero&    \MB - \MC^\tT \MA^{-1} \MC
        \end{smallmatrix}
\right]         
    \left[
        \begin{smallmatrix}
            \MI & \MA^{-1} \MC \\
            \Mzero^\tT & \MI
        \end{smallmatrix}
        \right]
        \in \R^{\ell+n\times\ell+n},          
    $$
    \end{linenomath*}
where $\MB \in \R^{n\times n}$ and $\MC \in \R^{\ell\times n}$.
The matrix $\MW/\MA \coloneqq \MB - \MC^\tT \MA^{-1} \MC \in \R^{n\times n}$ 
is called \emph{Schur complement} of $\MW$ and we have
    $\MW  \succeq 0$  
    if and only if $\MA \succ 0$ and 
    $\MW/\MA \succeq 0$.
   \end{proposition}

Now we can prove the main equivalence theorem.

\begin{theorem} \label{tho:equiv-S1}
    Problems 
    \eqref{eq:conv-tik_real} and \eqref{eq:conv-real-tik:2}
    are equivalent in the following sense:
    \begin{enumerate}[(i)]
        \item  
        If $(\hat {\Vx}, \hat {\Vr})$ solves \eqref{eq:conv-tik_real}, then
        $(\hat{\Vx}, \Re[\hat \Vr])$
        solves \eqref{eq:conv-real-tik:2}.
        \item 
        If $(\tilde{\Vx}, \tilde \Vell)$
        solves 
        \eqref{eq:conv-real-tik:2}, then
        $(\tilde{\Vx}, \tilde {\Vr})$
        with 
        $\tilde {\Vr}_{(n,m)} 
        = (\tilde \Vell_{(n,m)}, \Im[\Mcomplex(\tilde \Vx_m)^\tT \tilde \Vx_n ])^\tT$
        solves \eqref{eq:conv-tik_real}.
    \end{enumerate}
\end{theorem}

\begin{proof}
The Schur complement of $\MQ_{m,n}$
with respect to $\MI_2$ is given by 
\begin{linenomath*}
\begin{equation} \label{schur_1}
\MQ_{(n,m)}/ \MI_2 
=
\left[
\begin{smallmatrix}
            1 - \lvert \Vx_n \rvert^2 
            & \Vell_{(n,m)} - \langle \Vx_n , \Vx_m\rangle\\
            \Vell_{(n,m)} - \langle \Vx_n , \Vx_m\rangle
            & 1 - \lvert \Vx_m \rvert^2
\end{smallmatrix}
\right]
\end{equation}
\end{linenomath*}
On the other hand, permuting the fourth and fifth column/row 
of $\MP_{(n,m)}$,
we obtain 
\begin{linenomath*}
\begin{align}       \label{eq:perm-p}
        \tilde\MP_{(n,m)} 
        \coloneqq
        {\small
        \left[
        \scalebox{0.85}{
        $\begin{array}{cc|cc|cc}
            1 & 0 & \Re[\Vx_n] & \Re[\Vx_m] & -\Im[\Vx_n] & -\Im[\Vx_m] \\
            0 & 1 & \Im[\Vx_n] & \Im[\Vx_m] & \Re[\Vx_n] & \Re[\Vx_m] \\
            \hline
            \Re[\Vx_n] & \Im[\Vx_n] & 1 & \Re[\Vr_{(n,m)}] & 0 & \Im[\Vr_{(n,m)}] \\
            \Re[\Vx_m] & \Im[\Vx_m] & \Re[\Vr_{(n,m)}] & 1 & -\Im[\Vr_{(n,m)}] & 0 \\
            \hline
            -\Im[\Vx_n] & \Re[\Vx_n] & 0 & -\Im[\Vr_{(n,m)}] & 1 & \Re[\Vr_{(n,m)}] \\
            -\Im[\Vx_m] & \Re[\Vx_m] & \Im[\Vr_{(n,m)}] & 0 & \Re[\Vr_{(n,m)}] & 1
        \end{array}$}
        \right],
        } 
\end{align}
\end{linenomath*}
which has the Schur complement 
\begin{linenomath*}
\begin{align}
        \tilde\MP_{(n,m)} / \MI_2 =
        \left[
        \scalebox{0.64}{
        $\begin{array}{cc|cc}
            1 - \lvert \Vx_n \rvert^2 
            & \Re[\Vr_{(n,m)}] - \Re[\Mcomplex( \Vx_m)^\tT  \Vx_n ]
            & 0
            & \Im[\Vr_{(n,m)}] - \Im[\Mcomplex( \Vx_m)^\tT  \Vx_n ]
            \\
            \Re[\Vr_{(n,m)}] - \Re[\Mcomplex( \Vx_m)^\tT  \Vx_n ]
            & 1 - \lvert \Vx_m \rvert^2
            & - \Im[\Vr_{(n,m)}] + \Im[\Mcomplex( \Vx_m)^\tT  \Vx_n ]
            & 0
            \\
            \hline
            0
            & -\Im[\Vr_{(n,m)}] + \Im[\Mcomplex( \Vx_m)^\tT  \Vx_n ]
            & 1 - \lvert \Vx_n \rvert^2
            & \Re[\Vr_{(n,m)}] - \Re[\Mcomplex( \Vx_m)^\tT  \Vx_n ]
            \\
            \Im[\Vr_{(n,m)}] - \Im[\Mcomplex( \Vx_m)^\tT  \Vx_n ]
            & 0
            & \Re[\Vr_{(n,m)}] - \Re[\Mcomplex( \Vx_m)^\tT  \Vx_n ]
            & 1 - \lvert \Vx_m \rvert^2
            \label{eq:schur-p}
        \end{array}$}
        \right]\!\! .
\end{align}
\end{linenomath*}
    Recall that 
    $\Re[\Mcomplex( \Vx_m)^\tT  \Vx_n ] 
    = \langle \Vx_n, \Vx_m \rangle$.
    Using the Schur complements in \eqref{schur_1} and \eqref{eq:schur-p} 
    as well as 
    Proposition \ref{prop:schur-comp}, 
    we obtain the following:		
    if $(\hat {\Vx}, \hat {\Vr})$ solves \eqref{eq:conv-tik_real}, then 
    $(\hat{\Vx}, \hat \Vell)$ with $\hat \Vell \coloneqq \Re[\hat \Vr]$ is a feasible point of \eqref{eq:conv-real-tik:2}.
    Moreover, it holds by their definition that $\FK(\hat{\Vx}, \hat \Vell) = \FJ(\hat {\Vx}, \hat {\Vr})$.
    
    Conversely,	if $(\tilde{\Vx}, \tilde{\Vell})$
    solves 
    \eqref{eq:conv-real-tik:2}, then
    $(\tilde {\Vx}, \tilde {\Vr})$ with $\tilde {\Vr}$ defined in the theorem
    is a feasible point of \eqref{eq:conv-tik_real}.
    Further, we have $\FK(\tilde{\Vx}, \tilde \Vell) = \FJ(\tilde {\Vx}, \tilde {\Vr})$.
    Taking the minimizing property into account, we conclude  
    \begin{linenomath*}
    $$
    \FK(\tilde{\Vx}, \tilde \Vell) \le  
    \FK(\hat{\Vx}, \hat \Vell) 
    = 
    \FJ(\hat {\Vx}, \hat {\Vr}) \le \FJ(\tilde{\Vx}, \tilde {\Vr}) = \FK(\tilde{\Vx}, \tilde \Vell),
    $$
    \end{linenomath*}
    which is only possible if all values coincides.		
\end{proof}

Our previous considerations yield immediately the following corollary.

\begin{corollary} \label{cor:equiv-S1}
    Problems 
    \eqref{eq:conv-tik} and \eqref{eq:conv-real-tik:2}
    are equivalent in the following sense:
    \begin{enumerate}[(i)]
        \item  
        If $(\hat {x}, \hat {r})$ solves \eqref{eq:conv-tik}, then
        $(\hat{\Vx}, \Re[\hat r])$
        solves \eqref{eq:conv-real-tik:2}.
        \item 
        If $(\tilde{\Vx}, \tilde \Vell)$
        solves 
        \eqref{eq:conv-real-tik:2}, then
        $(\tilde{x}, \tilde {r})$
        with 
        $\tilde {r}_{(n,m)} 
        = \tilde \Vell_{(n,m)} + \I \Im[\overline{\tilde x}_m \tilde x_n ] $
        solves \eqref{eq:conv-tik_real}.
    \end{enumerate} 
\end{corollary}

\section{Generalization to Sphere-Valued Data} \label{sec:sd}

The main advantage of the real-valued, relaxed model \eqref{eq:conv-real-tik:2} is 
its simple generalization to higher dimensions,

which we derive in this section.
Defining the $(d-1)$-dimensional sphere as
$\sphere_{d-1} \coloneqq \{ \Vx \in \R^{d} :  \lVert \Vx \rVert = 1\}$,
we aim to determine the signal $\Vx \coloneqq (\Vx_n)_{n\in V} \in \sphere_{d-1}^N$ on $G$
from perturbed values $\Vy \coloneqq (\Vy_n)_{n\in V} \in (\mathbb S_{d-1})^N$.
Note that the following relaxation remains valid for $\Vy_n \in \R^d$.
We start with the original, nonconvex problem
\begin{linenomath*}
\begin{equation}    
    \label{eq:class-tik_real_rd}
    \argmin_{\Vx \in \sphere_{d-1}^N} 
    \sum_{n \in V} 
    \frac{w_n}{2} \, \lVert \Vx_n - \Vy_n \rVert^2
    + \smashoperator{\sum_{(n,m) \in E}}
    \frac{\lambda_{(n,m)}}{2} \, \lVert \Vx_n - \Vx_m \rVert^2,
\end{equation}
\end{linenomath*}
and its rewritten version
\begin{linenomath*}
\begin{equation}     
    \label{eq:ext-real-tik_rd}
    \argmin_{\Vx \in \sphere_{d-1}^N , \Vell \in \R^M} 
    \FK(\Vx,\Vell)
    \quad \text{s.t.} \quad
    \Vell_{(n,m)} = \langle \Vx_n, \Vx_m \rangle
\end{equation}
\end{linenomath*}
for all $(n,m) \in E$ with
\begin{linenomath*}
\begin{equation}
    \FK(\Vx,\Vell) 
    \coloneqq
    - \sum_{n \in V}
    w_n \, \langle \Vx_n, \Vy_n\rangle 
    - \smashoperator{\sum_{(n,m) \in E}} 
    \lambda_{(n,m)} \, \Vell_{(n,m)}
\end{equation}
\end{linenomath*}

similar to the embedding in the 2-dimensional Euclidean vector space.
Immediately,
Lemma~\ref{lem:3} extends to higher dimensions as follows.

\begin{lemma} \label{lem:4}
    Let $n,m \in V$ and $(n,m) \in E$.
    Then $\Vx_n, \Vx_m \in \sphere_{d-1}$ 
    and $\Vell_{(n,m)} = \langle \Vx_n, \Vx_m \rangle$
    if and only if
    the block matrix
    \begin{linenomath*}
    \begin{equation}    \label{eq:mat_real_d}
        \MQ_{(n,m)}
        \coloneqq
        \begin{bmatrix}
            \MI_d & \Vx_n & \Vx_m \\
            \Vx_n^\tT & 1 & \Vell_{(n,m)} \\
            \Vx_m^\tT & \Vell_{(n,m)} & 1
        \end{bmatrix}
        \in \R^{(d+2) \times (d+2)}
    \end{equation}
    \end{linenomath*}
    is positive semi-definite and has rank $d$. 
\end{lemma}

Incorporating Lemma~\ref{lem:4} and
relaxing the rank-$d$ constraint,
we obtain the new $d$-dimensional generalization:
\begin{linenomath*}
\begin{align}    
&\textbf{simplified relaxed real model}\\[1ex]
    \label{eq:conv-real-tik}
&    \argmin_{\Vx \in (\R^d)^N, \Vell \in \R^M} 
    \FK (\Vx, \Vell) 
    \quad \text{s.t.} \quad 
    \MQ_{(n,m)} \succeq 0
\end{align}
\end{linenomath*}
for all $(n,m) \in E$,
which is again a real-valued, convex optimization problem
and can thus be solved applying standard numerical methods
from convex analysis.
From a theoretical point of view,
the tightness of the simplified relaxed real model
to the original nonconvex problem 
remains an open issue,
i.e.,
having a solution of \eqref{eq:conv-real-tik},
it is unclear 
how to construct a solution of \eqref{eq:class-tik_real_rd}.
From a numerical point of view,
we observe in our experiments
that the calculated solutions of \eqref{eq:conv-real-tik}
immediately solve \eqref{eq:class-tik_real_rd}.

\section{Generalization to $\SO(3)$-Valued Data} \label{sec:sdc}

In some applications like electron backscatter tomography,
the measured data naturally lie on the 3d rotation group 
$\SO(3) \coloneqq \{ \MR \in \R^{3,3} : \MR^\tT \MR = \MI_3, \det(\MR)=1 \}$.
The matrices in $\SO(3)$ form a three-dimensional manifold
and can be parameterized using 
the rotation axis $\Vv \in \sphere_2$ and the rotation angle $\alpha \in [-\pi,\pi)$.
More precisely,
the rotation matrix corresponding to $\Vv$ and $\alpha$
is given by
\begin{linenomath*}
\begin{align}
    \label{eq:axis-ang-mat}
    &\MR(\Vv,\alpha) \coloneqq 
    \left[
    \scalebox{1.0}{$\begin{smallmatrix}
        (1 - \cos(\alpha)) \, \Vv_1^2 + \cos(\alpha)
        & (1- \cos(\alpha)) \, \Vv_1 \Vv_2 - \Vv_3 \sin(\alpha)
        & (1 - \cos(\alpha)) \, \Vv_1 \Vv_3 + \Vv_2 \sin(\alpha)
        \\
        (1 - \cos(\alpha)) \, \Vv_2 \Vv_1 + \Vv_3 \sin(\alpha)
        & (1 - \cos(\alpha)) \, \Vv_2^2 + \cos(\alpha)
        & (1 - \cos(\alpha)) \, \Vv_2 \Vv_3 - \Vv_1 \sin(\alpha)
        \\
        (1 - \cos(\alpha)) \, \Vv_1 \Vv_3 - \Vv_2 \sin(\alpha)
        & (1 - \cos(\alpha)) \, \Vv_3 \Vv_2 + \Vv_1 \sin(\alpha)
        & (1 - \cos(\alpha)) \, \Vv_3^2 + \cos(\alpha)
    \end{smallmatrix}$}
    \right].
\end{align}
\end{linenomath*}
Due to $\MR(\Vv,\alpha) = \MR(-\Vv,-\alpha)$,
the rotation angle may be restricted to $\alpha \in [0,\pi)$.
Moreover, $\SO(3)$ can be parameterized by unit quaternions, 
see for instance \cite{graef12},
which allow us to apply the derived methods to denoise $\SO(3)$-valued data.

Each quaternion $z \in \HH$ can be written as
$z \coloneqq z_1 + \I z_2 + \J z_3 + \K z_4$
with $z_1, z_2, z_3, z_4 \in \R$,
where the symbols $\I, \J, \K$ generalize the complex-imaginary unit.
The noncommutative multiplication of two quaternions is defined by the table
\begin{linenomath*}
\begin{align*}
    \I^2 = \J^2 = \K^2 = -1,&
    \quad
    \I \J = - \J \I = \K,\\
    \J \K = - \K \J = \I,&
    \quad
    \K \I = - \I \K = \J.
\end{align*}
\end{linenomath*}
The \emph{conjugate} of a quaternion is given by
$\bar z \coloneqq z_1 - \I z_2 - \J z_3 - \K z_4$,
and its \emph{norm} by
$\lvert z \rvert^2 \coloneqq \bar z z = z \bar z$.
The \emph{real} or \emph{scalar part} of $z$ is denoted by
$\Re[z] \coloneqq z_1$,
and its \emph{imaginary} or \emph{vector part} by
$\Im [z] \coloneqq \I z_2 + \J z_3 + \K z_4$.
Note that, 
different from the complex-imaginary part,
the quaternion-imaginary part contains the symbols $\I, \J, \K$.
The real components to the symbols are henceforth denoted by
$\Im_\I [z] \coloneqq z_2$,
$\Im_\J [z] \coloneqq z_3$,
$\Im_\K [z] \coloneqq z_4$.
The sphere of the unit quaternions is indicated by
$\Hsphere \coloneqq \{ z \in \HH : \lvert z \rvert = 1 \}$.

The action of rotations in $\SO(3)$ can now be identified
with the action of unit quaternions \cite{graef12}.
For this,
we identify the rotation axis $\Vv \in \sphere_2$ and angle $\alpha \in [-\pi,\pi)$
with the unit quaternion
\begin{linenomath*}
\begin{equation}
    \label{eq:quad-mat}
    q(\Vv,\alpha) 
    \coloneqq 
    \cos\bigl(\tfrac{\alpha}{2}\bigr) 
    + \sin\bigl(\tfrac{\alpha}{2}\bigr) \,
    (\I \Vv_1 + \J \Vv_2 + \K \Vv_3).
\end{equation} 
\end{linenomath*}
Notice that the real part of $q(\Vv,\alpha)$ is always nonnegative.
The rotation of $\Vx \in \R^3$ is given by
\begin{linenomath*}
\begin{equation*}
    \MR(\Vv,\alpha) \, \Vx 
    = (\Im_\I[\xi_{\Vx}], \Im_\J[\xi_{\Vx}], \Im_\K[\xi_{\Vx}])^\tT
\end{equation*}
\end{linenomath*}
with the quaternion
\begin{linenomath*}
\begin{equation*}
    \xi_{\Vx} 
    \coloneqq
    q(\Vv,\alpha) \,
    (\I \Vx_1 + \J \Vx_2 + \K \Vx_3) \, 
    \bar q(\Vv, \alpha),
\end{equation*}
\end{linenomath*}
i.e.\ the action of the rotation corresponds to a conjugation in the group 
of the unit quaternions.
Obviously, the unit quaternions $q$ and $-q$ correspond to the same rotation.
More precisely, 
the unit quaternions form a double cover of $\operatorname{SO}(3)$
by
\begin{linenomath*}
\begin{equation*}
    \Hsphere / \{-1,1\} \cong \operatorname{SO}(3),
\end{equation*}
\end{linenomath*}
where $q_1,q_2 \in \Hsphere$ are equivalent if $q_1 \bar q_2 \in \{-1,1\}$,
see \cite[Ch~III, §~10]{Bre93}.
The other way round,
the unit quaternion $q \in \Hsphere/\{-1,1\}$ corresponds to the rotation matrix
\begin{linenomath*}
\begin{align*}
\MR(q) \coloneqq
\left[\scalebox{1.0}{$\begin{smallmatrix}
    1 - 2 \Im_\J[q]^2 - 2 \Im_\K[q]^2 
    & 2 \Im_\I[q] \Im_\J[q] - 2 \Im_\K[q] \Re[q] 
    & 2 \Im_\I[q] \Im_\K[q] + 2 \Im_\J[q] \Re[q]
    \\
    2 \Im_\I[q] \Im_\J[q] + 2 \Im_\K[q] \Re[q] 
    & 1 - 2 \Im_\I[q]^2 - 2 \Im_\K[q]^2 
    & 2 \Im_\J[q] \Im_\K[q] - 2 \Im_\I[q] \Re[q]
    \\
    2 \Im_\I[q] \Im_\K[q] - 2 \Im_\J[q] \Re[q] 
    & 2 \Im_\J[q] \Im_\K[q] + 2 \Im_\I[q] \Re[q] 
    & 1 - 2 \Im_\I[q]^2 - 2 \Im_\J[q]^2
\end{smallmatrix}$}\right].
\end{align*}
\end{linenomath*}

Due to the parametrization of $\SO(3)$ by the unit quaternions,
we are especially interested in the generalization
of the complex-valued model \eqref{eq:class-tik} 
and its convex relaxation \eqref{eq:conv-tik} to the hypercomplex numbers $\HH$.
The aim is again to recover the signal $x \coloneqq (x_n)_{n \in V} \in \Hsphere^N$ on $G$
from noisy measurements $y \coloneqq (y_n)_{n \in V} \in \Hsphere^N$. 
Similarly to the complex-valued setting,
we consider the nonconvex problem
\begin{linenomath*}
\begin{equation}
  \label{eq:mod-quad-sph}
  \argmin_{x \in \Hsphere^N}
  \sum_{n \in V}
  \frac{w_n}{2} \,
  \lvert x_n - y_n \rvert^2
  +
  \smashoperator{\sum_{(n,m) \in E}}
  \frac{\lambda_{(n,m)}}{2} \,
  \lvert x_n - x_m \rvert^2,
\end{equation}
\end{linenomath*}
which is equivalent to
\begin{linenomath*}
\begin{equation}
  \label{eq:mod-quad-lin-ob}
  \argmin_{x \in \Hsphere^N, r \in \HH^M}
  \FJ(x,r)
  \quad\text{s.t.}\quad
  r_{(n,m)} =  \bar x_m x_n 
\end{equation}
\end{linenomath*}
for all $(n,m) \in E$ with
\begin{linenomath*}
\begin{equation*}
  \FJ(x,r)
  \coloneqq
  -
  \sum_{n \in V}
  w_n \, \Re[x_n \bar y_n]
  - \smashoperator{\sum_{(n,m) \in E}}
  \lambda_{(n,m)} \, \Re[r_{(n,m)}].
\end{equation*}
\end{linenomath*}
Since the quaternions are noncommutative, 
the order of multiplication in the constraint $r_{(n,m)} = \bar x_m x_n$ matters.

Many concepts of linear algebra generalize to quaternion matrices \cite{Zha97}.
A quaternion matrix $A \in \HH^{d \times d}$ 
is called \emph{Hermitian}
if $A = A^\tH$.
Such a matrix $A$
is positive semi-definite
if $z^\tH A z \ge 0$ for all $z \in \HH^d$.
Further, $A$ has rank one 
if all columns are \emph{right} linear dependent 
and all rows are \emph{left} linear dependent. 
Here the term \emph{left} and \emph{right}  indicate from which side
the quaternion coefficients are multiplied. 
In the Hermitian case, $A$ has rank one
if there exists  $z\in\HH^d$ such that $A = z z^\tH$.
Using a similar argumentation as in Lemma~\ref{lem:1},
we can rewrite the nonconvex constraints in terms of quaternion matrices.

\begin{lemma}
    \label{lem:8}
  Let $n,m \in V$ and $(n,m) \in E$.
  Then $x_n, x_m \in \Hsphere$ and $r_{(n,m)} = \bar x_m x_n$
  if and only if
  \begin{linenomath*}
  \begin{equation}
    \label{eq:P-quad}
    P_{(n,m)}
    \coloneqq
    \begin{bmatrix}
      1 & x_n & x_m \\
      \bar x_n & 1 & \bar r_{(n,m)} \\
      \bar x_m & r_{(n,m)} & 1
    \end{bmatrix}
    \in \HH^{3 \times 3}
  \end{equation}
  \end{linenomath*}
  is positive semi-definite and has rank one.
\end{lemma}

Incorporating Lemma~\ref{lem:8} into \eqref{eq:mod-quad-lin-ob},
and relaxing the rank-one constraint,
we obtain the new quaternion relaxation:
\begin{linenomath*}
\begin{align}
  \label{eq:conv-quad-tik}
  & \textbf{relaxed quaternion model}
  \notag \\
  &\argmin_{x \in \HH^N, r \in \HH^M}
  \FJ(x,r)
  \quad\text{s.t.}\quad
  P_{(n,m)} \succeq 0
\end{align}
\end{linenomath*}
for all $(n,m) \in E$.
This is the generalization of Condat's relaxed complex model 
\eqref{eq:conv-tik} for quaternion-valued data.
Using the approach as for complex data, 
we identify the quaternion $x_n \in \mathbb H$ with the vector 
$\Vx_n \coloneqq (\Re[x_n], \Im_\I[x_n], \Im_\J[x_n], \Im_\K [x_n])^\tT \in \mathbb R^4$
and  propose to solve instead of \eqref{eq:conv-quad-tik} 
again our new simplified relaxed real model \eqref{eq:conv-real-tik} with $d=4$,
i.e.\ we propose to solve the convex formulation:
\begin{linenomath*}
\begin{align}  
&\textbf{simplified relaxed real model}
\notag
\\[1ex]
    &\argmin_{\Vx \in (\R^4)^N, \Vell \in \R^M} 
    \FK (\Vx, \Vell) 
    \quad \text{s.t.} \quad 
    \MQ_{(n,m)} \succeq 0
    \label{eq:conv-real-tik-d4}
\end{align}
\end{linenomath*}
for all $(n,m) \in E$, where
\begin{linenomath*}
\begin{equation}    \label{eq:mat_real_d4}
        \MQ_{(n,m)}
        \coloneqq
        \begin{bmatrix}
            \MI_4 & \Vx_n & \Vx_m \\
            \Vx_n^\tT & 1 & \Vell_{(n,m)} \\
            \Vx_m^\tT & \Vell_{(n,m)} & 1
        \end{bmatrix}
        \in \R^{6 \times 6}.
\end{equation}
\end{linenomath*}
\vspace{0.2cm}
    
Note that this is exactly the simplified relaxed real model for $\sphere_3$-valued data.
Similarly to Corollary~\ref{cor:equiv-S1}, 
the following theorem establishes that
the relaxed quaternion model \eqref{eq:conv-quad-tik} is equivalent
to the simplified relaxed real one \eqref{eq:conv-real-tik-d4}. 

\begin{theorem} \label{the:equiv-quad}
    Problems 
    \eqref{eq:conv-quad-tik} and \eqref{eq:conv-real-tik-d4}
    are equivalent in the following sense:
    \begin{enumerate}[(i)]
        \item  
        If $(\hat x, \hat r)$ solves \eqref{eq:conv-quad-tik}, then
        $(\hat{\Vx}, \Re[\hat r])$
        solves \eqref{eq:conv-real-tik-d4}.
        \item 
        If $(\tilde{\Vx}, \tilde{\Vell})$
        solves 
        \eqref{eq:conv-real-tik-d4}, then
        $(\tilde x, \tilde r)$
        with 
        $\tilde r_{(n,m)} 
        = \tilde \Vell_{(n,m)} 
        + \Im[\overline{\tilde x}_m \tilde x_n]$
        solves \eqref{eq:conv-quad-tik}.
      \end{enumerate}
\end{theorem}

\begin{proof}
  See Appendix.
\end{proof}

Returning to 3d rotations,
we want to employ the simplified relaxed real model \eqref{eq:conv-real-tik-d4}
to denoise $\SO(3)$-valued signals.
More precisely,
having given perturbed data $(\MY_n)_{n\in V} \in (\SO(3))^N$,
we are looking for a smoothed signal $(\MX_n)_{n\in V} \in (\SO(3))^N$.
As discussed above,
the central idea is to use the double cover $\SO(3) \cong \Hsphere / \{-1,1\}$
to lift the given data $(\MY_n)_{n \in V}$ 
to a quaternion signal $(y_n)_{n\in V} \in \Hsphere^N$.
The main issue is here which sign of the corresponding quaternion should we choose?
Assuming that $(\MY_n)_{n \in V}$ originates from a \enquote{smooth} signal on $G$,
we expect the Frobenius norm
\begin{linenomath*}
\begin{equation}
    \lVert\MY_n - \MY_m\rVert_\tF^2  = 8(1-\Re[y_n \bar y_m ]^2)
\end{equation}
\end{linenomath*}
to be relatively small for all $(n,m) \in E$. 
The identity can be found in \cite[Eq. (27)]{HU2009} for instance.
Transferring the smoothness assumption,
we thus assume that $\Re[ y_n \bar  y_m ] \gg 0$
for all $(n,m) \in E$.
Starting from an arbitrary lifting of $\MY_1$,
we lift the remaining data $\MY_n$ such that $\Re[y_n \bar y_m] \ge 0$
for $(n,m) \in E$.
This lifting process is always successful
for any signal on a line-graph or tree.
During our numerical simulations,
we do not face any inconsistencies in
finding a global lifting
on well-connected graphs 
like the image graph.
Using the simplified real model \eqref{eq:conv-real-tik-d4}
and Theorem~\ref{the:equiv-quad},
we determine a smoothed signal $(x_n)_{n\in V} \in \Hsphere^N$,
which is retracted by $\MX_n \coloneqq \MR(x_n)$.

\begin{remark}[$\sphere_2$-Valued Data]\label{rem_1}
    Note that purely imaginary quaternions, 
    i.e.\ $z \in \HH$ with $\Re[z]=0$,
    can be identified with three-dimensional vectors.
    Therefore, 
    the restriction of $\Hsphere$ 
    to the purely imaginary quaternions
    can be identified with $\sphere_2$.
    A variation of Lemma~\ref{lem:8} shows that
    $x_n,x_m \in \Hsphere$ with $\Re[x_n] = \Re[x_m] = 0$ 
    and $r_{(n,m)}=\bar x_m x_n$
    if and only if
    \begin{linenomath*}
    \begin{equation}
        \label{eq:P-pure-quad}
        P'_{(n,m)}
        \coloneqq
        \begin{bmatrix}
          1 & \Im[x_n] & \Im[x_m] \\
          -\Im[x_n] & 1 & \bar r_{(n,m)} \\
          -\Im[x_m] & r_{(n,m)} & 1
        \end{bmatrix}
        \in \HH^{3 \times 3}
    \end{equation}
    \end{linenomath*}
    is positive semi-definite with rank one. 
    The corresponding \emph{relaxed purely quaternion model}
    \begin{linenomath*}
    \begin{equation}
      \label{eq:conv-pure-quad-tik}
      \argmin_{x \in \HH^N, r \in \HH^M}
      \FJ(x,r)
      \quad\text{s.t.}\quad
      P'_{(n,m)} \succeq 0
      \quad\text{for all} \;
      (n,m) \in E
    \end{equation}
    \end{linenomath*}
    becomes completely independent of the real parts $\Re[x_n]$,
    $n \in V$. 
    Setting these real parts to zero,
    we can use the relaxed quaternion problem for denoising $\sphere_2$-valued data.
    Furthermore,
    in line with Theorem~\ref{the:equiv-quad},
    the relaxed purely quaternion problem
    coincides with model \eqref{eq:conv-real-tik} for $d=3$.
    A similar approach has also been proposed by Condat 
    in \cite{condat_3D}.
\end{remark}

\begin{remark}[Octonion-Valued Data]\label{rem_3}
  Replacing the quaternions by octonions or 
  more general hypercomplex numbers,
  we may generalize \eqref{eq:conv-quad-tik} to more general spheres.
  Since the octonions are nonassociative,
  there exists however no real-valued matrix representation,
  which preserve addition and multiplication.
  Therefore,
  Lemma~\ref{lem:2} will not carry over to octonions,
  and the relation between
  the resulting relaxed octonion model
  and our simplified relaxed real model with $d=8$ proposed in Section~\ref{sec:sd}
  remains unclear.
\end{remark}

\section{Algorithm}\label{sec:algs}
For numerical simulations,
we solve the derived relaxed minimization problems by applying
the Alternating Directions Methods of Multipliers (ADMM) \cite{BPC+10,PB13}.
Exemplary, we consider the simplified relaxed 
$d$-dimensional real model \eqref{eq:conv-real-tik},
which can be rewritten as
\begin{linenomath*}
\begin{align}   
    \label{eq:num_real}
    \smashoperator[l]{\argmin_{\substack{\Vx \in (\R^d)^N,
    \Vell \in \R^M, \\
    \MU \in (\R^{(d+2) \times (d+2)})^M}}} 
    \!\!\!\!\FK(\Vx,\Vell) + \FG(\MU) 
    \quad\text{s.t.}\quad
    \FQ(\Vx,\Vell) = \MU,
\end{align}
\end{linenomath*}
where the linear map $\FQ \colon (\R^{d})^N \times \R^M 
\to (\R^{(d+2) \times (d+2)})^M$ is given by
\begin{linenomath*}
\begin{align}
    \FQ(\Vx,\Vell) 
    \coloneqq
    (\MQ_{(n,m)} - \MI_{d+2})_{(n,m) \in E},
\end{align}
\end{linenomath*}
and the positive semi-definite constraint is encoded in
\begin{linenomath*}
\begin{align}
    \FG(\MU) 
    \coloneqq
    \sum_{(n,m)\in E} \iota_{\cC}(\MU_{(n,m)}),
\end{align}
\end{linenomath*}
where $\iota_{\cC}$ denotes the \emph{indicator function} 
that is $0$ on $\cC$ and $+\infty$ otherwise,
and where
\begin{linenomath*}
$$
\cC \coloneqq \{ \MA \in \R^{(d+2)\times (d+2)} 
: \MA = \MA^\tT, \MA \succeq - \MI_{d+2} \}.
$$
\end{linenomath*}
For \eqref{eq:num_real} and $\rho > 0$,
ADMM reads as 
\begin{linenomath*}
\begin{align}
    (\Vx^{(k+1)}, \Vell^{(k+1)}) & \coloneqq 
    \argmin_{\Vx \in (\R^d)^N, \Vell \in \R^M}
    \FK(\Vx, \Vell)
    + \frac{\rho}{2} \, \lVert \FQ(\Vx, \Vell) - \MU^{(k)} + \MZ^{(k)}\rVert^2,
    \label{eq:admm:1}
    \\
    \MU^{(k+1)} & \coloneqq 
    \argmin_{\MU \in (\R^{(d+2) \times (d+2)})^M}\!\!\!\!
    \FG(\MU) 
    + \frac{\rho}{2} \,
    \lVert \FQ(\Vx^{(k+1)}, \Vell^{(k+1)}) - \MU + \MZ^{(k)} \rVert^2,
    \label{eq:admm:2}
    \\
    \MZ^{(k+1)} & \coloneqq
    \MZ^{(k)} + \FQ(\Vx^{(k+1)}, \Vell^{(k+1)}) - \MU^{(k+1)},
    \label{eq:admm:3}
\end{align}
\end{linenomath*}
where $\lVert \MU \rVert^2 \coloneqq \sum_{(n,m) \in E} \lVert \MU_{(n,m)} \rVert^2_\tF$
and $\lVert \cdot \rVert_\tF$ denotes the \emph{Frobenius norm}, see also
\cite{BurSawSte17}.
Note that we could alternatively define $\mathcal Q$ 
as a mapping to the linear subspace of symmetric matrices in 
$\mathbb R^{ (d+2) \times (d+2) }$.
It remains to give explicit formulas for the first and second update step.

\begin{theorem}
\label{tho:problemI}
    For $\rho > 0$,
    the solution of \eqref{eq:admm:1} is given by
    \begin{linenomath*}
    \begin{align*}
        \Vx_{n}^{(k+1)}
        &= 
        \frac{1}{2 \, \nu_n}
        \biggl(
        (\FQ_{\Vx}^*(\MU^{(k)} - \MZ^{(k)}))_{n}
        + \frac{1}{\rho} \, w_n \Vy_{n}
        \biggr),
        \\
        \Vell_{(n,m)}^{(k+1)}
        &= 
        \frac{1}{2}
        \biggl(
        (\FQ_{\Vell}^*(\MU^{(k)} - \MZ^{(k)}))_{(n,m)}
        + \frac{1}{\rho} \, \lambda_{(n,m)}
        \biggr),
    \end{align*}
    \end{linenomath*}
    where 
    $\nu_n 
    \coloneqq 
    \lvert \{ (n,m) \in E \}\rvert
    + \lvert \{ (m,n) \in E \}\rvert$
    counts the edges starting or ending in $n$, and the restrictions
    $\FQ_{\Vx}^* \coloneqq \FQ^*|_{(\R^{d})^N}$ and
$\FQ_{\Vell}^* \coloneqq \FQ^*|_{\R^{M}}$ of 
the adjoint operator $\FQ^*$ with respect to the component spaces
are given for $n \in V$ and $i=1,\dots, d$ by
\begin{linenomath*}
\begin{align}
    \hspace{-0.6cm}\bigl((\FQ_{\Vx}^* (\MU))_{n}\bigr)_i \!
    &=\!
    \Bigl[
    \smashoperator[r]{\sum_{(n,m) \in E}} \,
    (\MU_{(n,m)})_{i,d+1}
    \!+\! 
    (\MU_{(n,m)})_{d+1,i}
    \Bigr] +
    \Bigl[
    \smashoperator[r]{\sum_{(m,n) \in E}} \,
    (\MU_{(n,m)})_{i,d+2}
    \!+\! 
    (\MU_{(n,m)})_{d+2,i}
    \Bigr]  \label{adj_x}
\end{align}
\end{linenomath*}
and for $(n,m) \in E$ by
\begin{linenomath*}
\begin{equation}\label{adj_l}
    (\FQ_{\Vell}^* (\MU))_{(n,m)}
    = (\MU_{(n,m)})_{d+1,d+2} + (\MU_{(n,m)})_{d+2,d+1}.
\end{equation}
\end{linenomath*}
\end{theorem}

\begin{proof}
    Setting the gradient with respect to $\Vx$ and $\Vell$ 
    of the objective in \eqref{eq:admm:1} to zero,
    we obtain
    \begin{linenomath*}
    \begin{align}
        \frac{1}{\rho} (w_n \Vy_n )_{n \in V}
        + \FQ_{\Vx}^*( \MU^{(k)} - \MZ^{(k)})
        &=
        \FQ_{\Vx}^*( \FQ(\Vx, \Vell)),
        \\
        \frac{1}{\rho} \lambda
        + \FQ_{\Vell}^*(\MU^{(k)} - \MZ^{(k)})
        &=
        \FQ_{\Vell}^*( \FQ(\Vx, \Vell)).
    \end{align}
    \end{linenomath*}
    Due to 
    $(\FQ_{\Vx}^*(\FQ(\Vx, \Vell)))_{n} = 2 \nu_n \, \Vx_{n}$
    and 
    $(\FQ_{\Vell}^*(\FQ(\Vx, \Vell)))_{(n,m)} = 2 \Vell_{(n,m)}$,
    we obtain the assertion.
\end{proof}

The second ADMM step \eqref{eq:admm:2} can be separately computed
for single edges $(n,m) \in E$
and consists
in finding $\MU_{(n,m)}$ 
\begin{linenomath*}
\begin{equation}
    \argmin_{\MU_{(n,m)} \in \R^{(d+2)\times(d+2)}}
    \iota_{\cC}(\MU_{(n,m)}) 
    + \frac{\rho}{2} \,
    \lVert (\FQ(\Vx^{(k+1)}, \Vell^{(k+1)}))_{(n,m)}
    - \MU_{(n,m)} + \MZ_{(n,m)}^{(k)} \rVert^2_\tF,
    \label{eq:admm2-proj}
\end{equation}
\end{linenomath*}
which is the projection of 
$(\FQ(\Vx^{(k+1)}, \Vell^{(k+1)}))_{(n,m)} + \MZ_{(n,m)}^{(k)}$
onto $\cC$.

\begin{theorem}
    \label{tho:problemII}
    Let $\MA = \MV\MSigma \MV^\tT \in \R^{(d+2) \times (d+2)}$
    be symmetric with $\MV^\tT\MV = \MI_{d+2}$ 
    and diagonal matrix $\MSigma \in \R^{(d+2)\times (d+2)}$.
    Then it holds
    \begin{linenomath*}
    \begin{equation} \label{proj}
        \proj_{\cC} (\MA) = \MV\hat\MSigma \MV^\tT,
    \end{equation}
    \end{linenomath*}
    where $\hat\MSigma \coloneqq 
    \max(\mathrm{diag}(\MSigma), -1)$ is meant componentwise.
\end{theorem}

\begin{proof}
    For any  $\hat\MA \in \mathcal C$,
    we have
    \begin{linenomath*}
    $$
        \lVert\hat\MA - \MA\rVert_\tF  
         = \lVert\underbracket{\MV^\tT \hat\MA \MV}_{\eqqcolon \MB}  - \MSigma\rVert_\tF.
    $$ 
    \end{linenomath*}
    Since $\hat\MA \succeq -\MI_{d+2}$,
   we  know that $\MB_{ii} \ge -1$.
   Hence we obtain
   \begin{linenomath*}
    \begin{align}
        \lVert\hat\MA - \MA\rVert_\tF^2             
         &\geq \sum_{i=1}^{d+2} |\MB_{ii} - \MSigma_{ii}|^2 
         \geq \smashoperator{\sum_{i \in \{ i : \MSigma_{ii} < -1\}}} |\MB_{ii} - \MSigma_{ii}|^2
        \geq \smashoperator{\sum_{i \in \{ i : \MSigma_{ii} < -1\}}} |\MSigma_{ii} + 1 |^2.
    \end{align}
    \end{linenomath*}
    Now the lower bound is exactly attained  
    for $\hat\MA = \MV \hat \MSigma \MV^\tT$ with 
    $\hat \MSigma = \max(\mathrm{diag}(\MSigma), -1)$.
\end{proof}

Note that
the special structure of $\MQ_{(n,m)} - \MI_{d+2}$ 
is not preserved by the projection onto $\cC$.
Summarizing the update steps in 
Theorem~\ref{tho:problemI} and \ref{tho:problemII},
we can solve our simplified relaxed real model \eqref{eq:conv-real-tik}
using the following ADMM scheme, 
whose convergence is here ensured by \cite[Cor~28.3]{bauschke}.

\begin{algorithm}[ADMM for Solving \eqref{eq:conv-real-tik}]
    \label{alg:1}
    ~\newline
    \emph{Input:} 
    $\Vy \in (\R^d)^N$,
    $w \in \R^N_+$,
    $\lambda \in \R^{M}_+$, 
    $\rho > 0$.
    \newline
    \emph{Initiation:}
    $\Vx^{(0)} \leftarrow 0 \in (\R^d)^N$,
    $\Vell^{(0)} \leftarrow 0 \in \R^M$,
    \\
    $\MU^{(0)} \leftarrow \MZ^{(0)} \leftarrow 0 \in (\R^{(d+2)\times(d+2)})^M$,
    \newline
    \emph{Iteration:} 
    For $k = 0,1,2,3,\dots$ until convergence:\\
    \noindent
    \hspace*{15pt}
    $\Vx_{n}^{(k+1)}
    \leftarrow 
    \frac{1}{2 \nu_n}
    \bigl(
    (\FQ_{\Vx}^*(\MU^{(k)} - \MZ^{(k)}))_{n}
    + \frac{1}{\rho} \, w_n \Vy_{n}
    \bigr)
    $
    for all $n \in V$ (see Thm~\ref{tho:problemI}).
    \\
    \hspace*{15pt}
    $
    \Vell_{(n,m)}^{(k+1)}
    \leftarrow
    \frac{1}{2}
    \bigl(
    (\FQ_{\Vell}^*(\MU^{(k)} - \MZ^{(k)}))_{(n,m)}
    + \frac{1}{\rho} \, \lambda_{(n,m)}
    \bigr)$
    for all $(n,m) \in E$ (see Thm~\ref{tho:problemI}).
    \\
    \hspace*{15pt}
    $\MU^{(k+1)}_{(n,m)}
    \leftarrow 
    \proj_{\cC} \bigl( (\FQ(\Vx^{(k+1)},\Vell^{(k+1)}))_{(n,m)} 
    + \MZ^{(k)}_{(n,m)} \bigr)$
    for all $(n,m) \in E$ (see Thm~\ref{tho:problemII}).
    \\
    \hspace*{15pt}
    $\MZ^{(k+1)} 
    \leftarrow
    \MZ^{(k)} + \FQ(\Vx^{(k+1)}, \Vell^{(k+1)}) - \MU^{(k+1)}$.
    \newline
    \emph{Output:} 
    $(\tilde \Vx, \tilde \Vell)$ solving \eqref{eq:conv-real-tik}.
\end{algorithm}

Similar schemes can be derived for the relaxed complex and quaternion problem.

\section{Numerical Experiments} \label{sec:numerics}

During the following numerical simulation, 
we give a proof-of-the-concept 
for denoising sphere and $\SO(3)$-valued data.
We compare our simplified relaxed real model \eqref{eq:conv-real-tik:2}
with Condat's relaxed complex model \eqref{eq:conv-tik} 
for circle-valued data.
Here, the proposed ADMM in Algorithm \ref{alg:1} 
shows a significant faster convergence behaviour
than the originally proposed 
\textit{Proximal Method of Multipliers} (PMM) 
in \cite{condat_1D2D},
which depends on two parameters $\tau$ and $\sigma$.
Furthermore, we apply our simplified relaxed real method 
for hue and chromaticity denoising in imaging.
The employed algorithms are implemented\footnote{
The code is available at GitHub: \url{https://github.com/JJEWBresch/relaxed_tikhonov_regularization}.} 
in Python~3.11.4 using Numpy~1.25.0 and Scipy~1.11.1.
The experiments are performed on an off-the-shelf iMac 2020
with Apple M1 Chip (8‑Core CPU, 3.2~GHz) and 8~GB RAM.

\subsection{$\sphere_1$-Valued Data}        \label{sec:s1-data}

We start with a synthetic, smooth signal $(\Vx_n)_{n \in V}$
on the line graph.
More precisely, we consider the circle-valued, 
one-dimensional signal
in Figure~\ref{fig:func_value_S1_1D}.
The synthetic noisy observation $(\Vy_n)_{n \in V}$
are generated using the \textit{von Mises--Fisher} distribution 
by 
\begin{linenomath*}
$$
\Vy_n \sim \mathcal N_{\text{vMF}}(\Vx_n, \kappa) \quad \text{for all} \quad n \in V,
$$
\end{linenomath*}
where $\kappa>0$ is the so-called capacity.
To denoise the generated measurements,
we apply PMM ($\tau \coloneqq 0.1$, $\sigma \coloneqq (4 \tau)^{-1}$)
on the relaxed complex model \eqref{eq:conv-tik},
ADMM ($\rho \coloneqq 3$) 
on the relaxed real model \eqref{eq:conv-tik_real},
and compare these with ADMM ($\rho \coloneqq 3$) in Algorithm \ref{alg:1} 
for the simplified relaxed real model \eqref{eq:conv-real-tik:2},
where the regularization parameters 
are chosen as $w_n \coloneqq 1$ and $\lambda_{(n,m)} \coloneqq 25$.
Starting all algorithms from zero, 
we observe convergence 
to the same limit, which is shown in Figure~\ref{fig:func_value_S1_1D}.

Additionally, 
we compare the numerical solution of the relaxed models 
with the one of the TV solver CPPA-TV 
(cyclic proximal point algorithm for TV of circle-valued data)
from \cite{BerLauSteWei14}.
Since the underlying signal is rather smooth,
we observe the well-known staircase effects 
of the TV regularization.
For this comparison,
we stop both algorithm---%
ADMM on \eqref{eq:conv-real-tik:2}
and CPPA-TV---%
if the residuum
(the difference between subsequent iterates in the 2-norm)
is smaller than $10^{-4}$.
Although our method has to handle positive semi-definite matrices,
the computation time amounts to 3~seconds,
whereas CPPA-TV,
which solves the TV regularization directly on the circle,
needs 25~seconds.
To compare the convergence speed 
with respect to the three different relaxed Tikhonov models
in more detail, 
we run those algorithms in a second experiment 
for 600 iterations 
and determine the time after which the objective 
remains in a small $\epsilon \coloneqq 10^{-5}$ 
neighbourhood around the limiting value.
These times are recorded in Table~\ref{tab:func_value_S1_1D}.
For the relaxed complex model \eqref{eq:conv-tik},
ADMM shows a significant speed-up compared with PMM. 
Moreover, switching to our simplified relaxed real model 
gives an additional acceleration.
Notice that the solution $(\tilde \Vx_n)_{n \in V}$
of \eqref{eq:conv-real-tik:2} fulfils $\lVert\tilde \Vx_n\rVert \leq 1$,
since $\MQ_{(n,m)} \succeq 0$
and thus 
\begin{linenomath*}
$$
\det\Bigl(\begin{smallmatrix}
    \MI & \tilde\Vx_n \\
    \tilde\Vx_n^{\tT} & 1
\end{smallmatrix}\Bigr) 
= 1 - \lVert\tilde \Vx_n\rVert^2 
\geq 0
$$
\end{linenomath*}
for all $(n,m) \in E$.
For the relaxed complex models, there holds an analogous observation.
Determining the mean of $1 - \lVert\tilde \Vx_n\rVert$ over $n \in V$
as well as of $1 - \lvert x_n \rvert$ 
for the numerical real and complex solutions, see Table~\ref{tab:func_value_S1_1D},
we observe convergence to the circle,
i.e.\ the numerical solutions solve the original 
(unrelaxed) problem \eqref{eq:class-tik_real_1}.
For the simulation in Figure~\ref{fig:func_value_S1_1D},
the mean distance to the circle over the run time
is reported in Figure~\ref{fig:sphereconvsignal}.

\begin{figure}
\includegraphics[width=\linewidth, clip=true, trim=100pt 5pt 90pt 20pt]{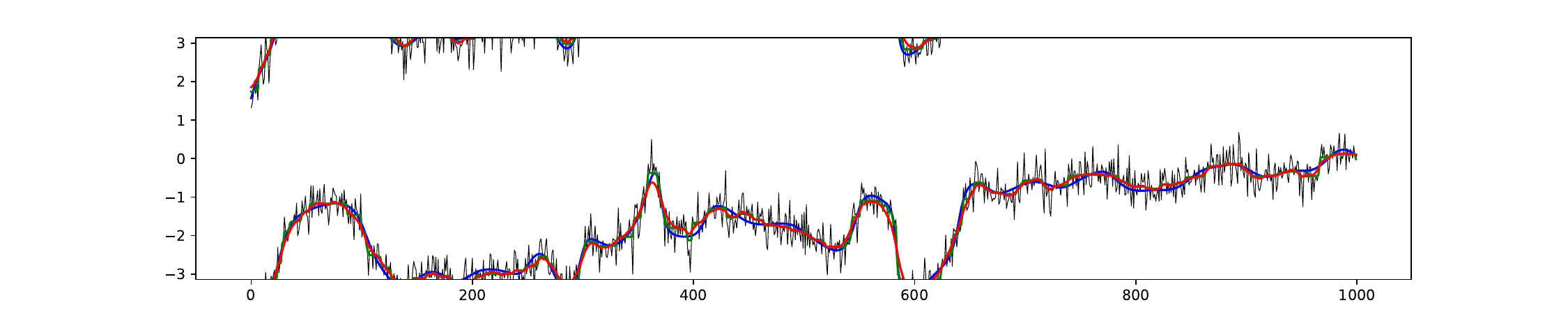}
\caption{
Ground truth (blue), 
noisy observation with $\kappa = 10$ (black),
numerical solution by Alg.~\ref{alg:1} (red)
with $\lambda \equiv 25$ and $\rho = 3$
of the line graph signal of length $N = 1000$ in Section~\ref{sec:s1-data}
in comparison with the solution obtained by CPPA-TV (green)
with regularization parameter $\lambda \equiv 0.8$ and $\lambda_0 = \pi$.
The $\sphere_1$-values are represented by their angles in  $[-\pi,\pi)$.}
\label{fig:func_value_S1_1D}
\end{figure}

\begin{SCfigure}[1]
\caption{\label{fig:sphereconvsignal}
Mean distance to the circle over the computation time
    for the experiment in Figure~\ref{fig:func_value_S1_1D}.}
\includegraphics[width=0.66\linewidth, clip=true, trim=55pt 5pt 70pt 25pt]{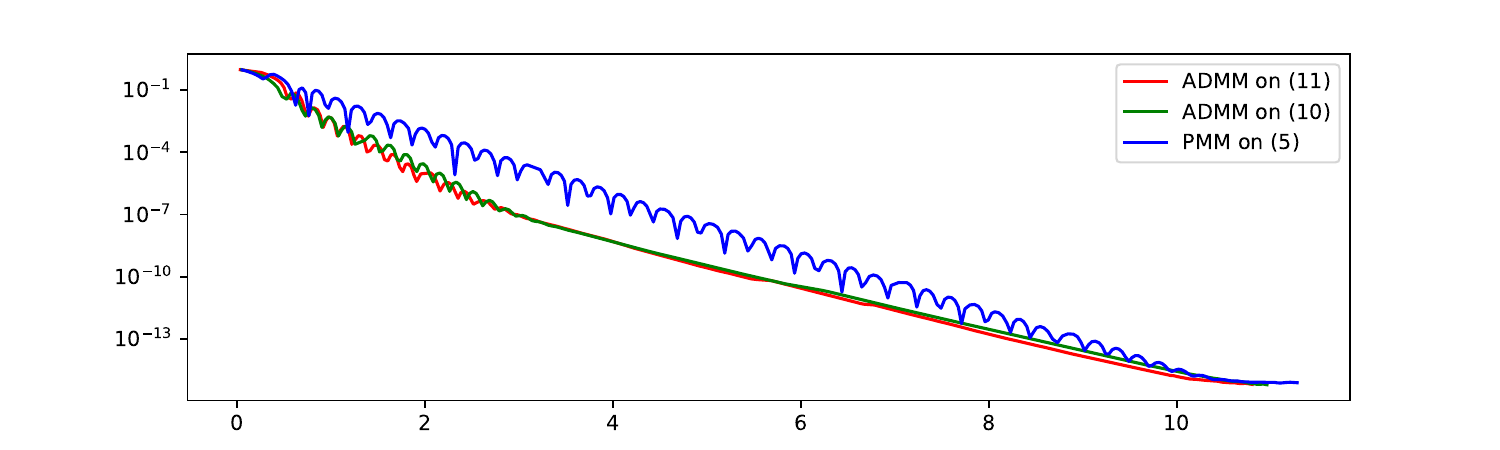}
\end{SCfigure}

\begin{SCtable}[2]
\caption{
The mean distance to the circle,
computation time, 
and iteration numbers
for the line graph signal in Section~\ref{sec:s1-data}
with $\kappa = 10$.
The recorded values are averages over 50 randomly generated ground truths. 
}
\centering\footnotesize
\begin{tabular}{l c c c}
\toprule
 	& mean & time & iter. \\
\midrule
PMM on \eqref{eq:conv-tik} & $10^{-9}$ & $2.10$ & $201$ \\
ADMM on \eqref{eq:conv-tik_real} & $10^{-12}$ &  $1.83$ & $182$ \\
ADMM on \eqref{eq:conv-real-tik:2} & $10^{-13}$ & $1.77$ & $181$\\
\bottomrule
\end{tabular}
\label{tab:func_value_S1_1D}
\end{SCtable}

\begin{figure}
\begin{minipage}[c]{0.7\textwidth}
    \includegraphics[width=0.90\linewidth, clip=true, trim=135pt 20pt 487pt 40pt]{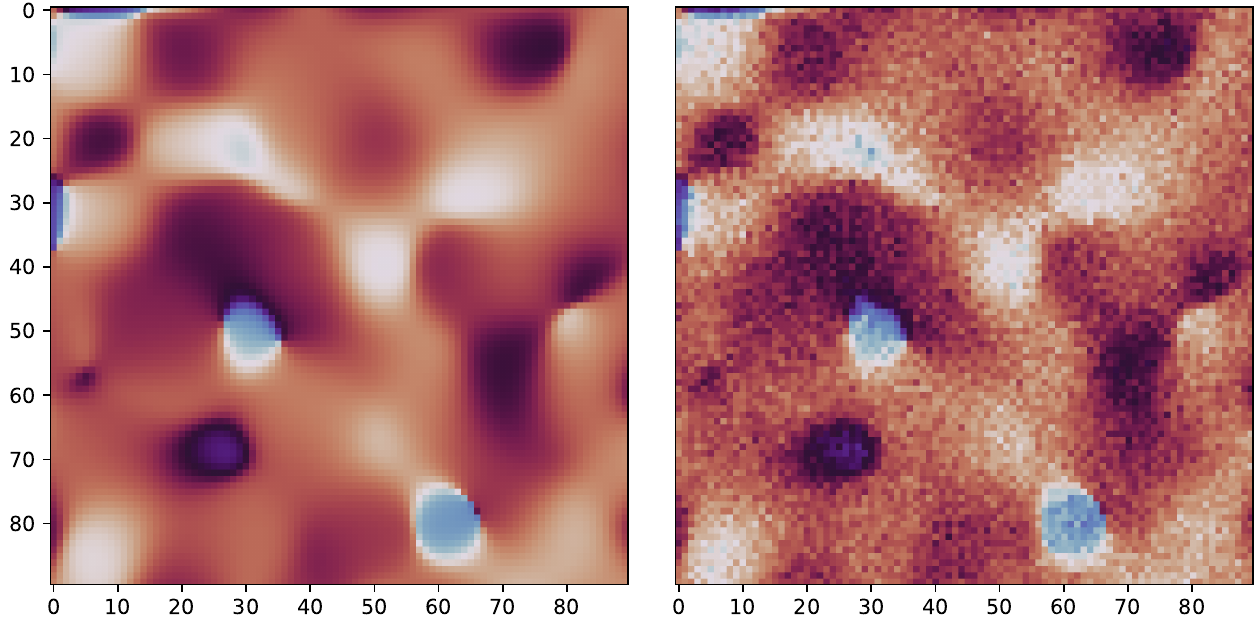}
    \includegraphics[width=0.072\linewidth, clip=true, trim=1045pt 20pt 135pt 40pt]{diagrams_circle/ADMMred_1_2500_0.0001.pdf} \\
    \vspace{0.075cm}
    \includegraphics[width=0.445\linewidth, clip=true, trim=745pt 20pt 185pt 40pt]{diagrams_circle/ADMMred_1_2500_0.0001.pdf}
    \includegraphics[width=0.43\linewidth, clip=true, trim=749pt 20pt 190pt 40pt]{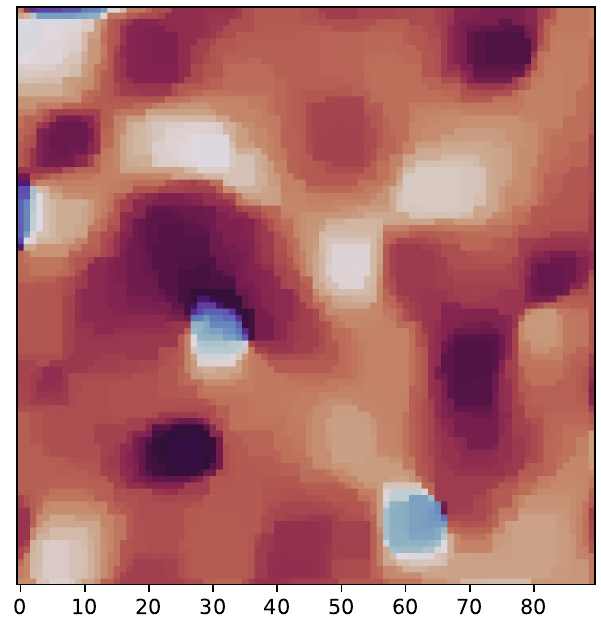}
    \includegraphics[width=0.073\linewidth, clip=true, trim=1045pt 20pt 135pt 40pt]{diagrams_circle/ADMMred_1_2500_0.0001.pdf} \\
\end{minipage}
\hfill
\begin{minipage}[c]{0.28\textwidth}
\vspace{-1.25cm}
    \caption{\textit{Top row}: Ground truth (left) and
noisy observation with $\kappa = 20$ (right).
\textit{Bottom row}: Numerical solution (left) using Alg.~\ref{alg:1}
on the $90 \times 90$ pixel image graph signal in Section~\ref{sec:s1-data} 
with $w_n \equiv 1, \lambda \equiv 1$ and $\rho = 3$.
The solution comes with
an restoration error (RMSE) $7.627\cdot10^{-2}$,
mean distance to the sphere $7.834\cdot10^{-5}$.
Comparison with CPPA-TV (right)
with $\lambda \equiv 0.3$ and $\lambda_0 = \pi$.
The solution comes with an restoration error (RMSE) $8.702\cdot10^{-2}$.
The colour map corresponds to the angles of the $\sphere_1$-values.}
\label{fig:func_value_S1_2D}
\end{minipage}
\end{figure}

\begin{SCfigure}[1]
\caption{\label{fig:sphereconv}
Mean distance to the circle over the computation time
    for the experiment in Figure~\ref{fig:func_value_S1_1D}.}
\includegraphics[width=0.66\linewidth, clip=true, trim=60pt 5pt 60pt 25pt]{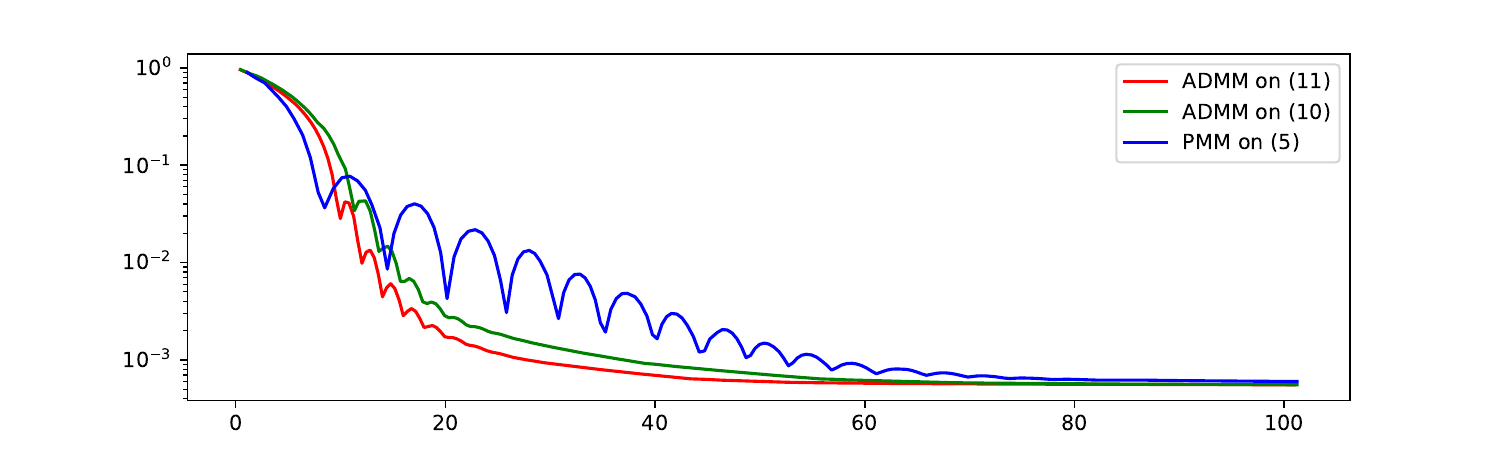}
\end{SCfigure}

\begin{SCtable}[2]
\caption{
The mean distance to the circle,
computation time, 
and iteration numbers
for the image graph signal in Section~\ref{sec:s1-data}
with $\kappa = 20$.
The recorded values are averages over 50 randomly generated ground truths. 
}
\centering\footnotesize
\begin{tabular}{l c c c c }
\toprule
		 & mean & time & iter. \\
\midrule
PMM on \eqref{eq:conv-tik} & $10^{-4}$ & $747.18$ & $5536$\\
ADMM on \eqref{eq:conv-tik_real} & $10^{-4}$ & $389.01$ & $2457$ \\
ADMM on \eqref{eq:conv-real-tik:2} & $10^{-4}$ & $301.08$ & $1943$\\
\bottomrule
\end{tabular}
\label{tab:func_value_S1_2D}
\end{SCtable}

For $\sphere_1$-valued images, 
the (\emph{image}) graph $G$ consists of the image pixels,
which are connected to their vertical and 
horizontal neighbours. 
As before, we generate synthetic observation
using the von Mises--Fisher distribution.
The synthetic observations are denoised 
using PMM ($\tau \coloneqq 0.1$, $\sigma \coloneqq (8 \tau)^{-1}$)
on the relaxed complex model \eqref{eq:conv-tik},
ADMM ($\rho \coloneqq 3$)
on the relaxed real model \eqref{eq:conv-tik_real},
as well as ADMM ($\rho \coloneqq 3$) on the simplified relaxed real model \eqref{eq:conv-real-tik:2},
where the regularization parameters 
are chosen as $w_n \coloneqq 1$ and $\lambda_{(n,m)} \coloneqq 1$.
The generated data together with its denoised version,
which nearly coincides for all three methods,
are shown in Figure~\ref{fig:func_value_S1_2D}.  
Similarly to the previous example,
we compare our relaxed models with CPPA-TV.
The typical staircase effects can again be observed
together with a larger restoration error (root mean squared error (RMSE))
compared to our model.
The stopping criteria is chosen as for the line-graph example above.
The computation time for our approach amounts to 65 seconds 
and for CPPA-TV to 367 seconds.
To compare the convergence speed with respect 
to the three different relaxed Tikhonov models in more details, 
we run the algorithms in a second experiment 
for 6000 iterations and determine the time after which 
the objective stays in an $\epsilon \coloneqq 10^{-3}$ neighborhood 
around the limiting value.
These times are recorded in Table~\ref{tab:func_value_S1_2D}.
Differently from the line-graph setting,
all methods show much slower numerical convergence.
Similarly, 
the convergence of $\Vx_n$ to $\sphere_1$ and $x_n$ to $\Csphere$ is slowed down
but can be observed after the final iteration,
see Figure~\ref{fig:sphereconv}. 
This behaviour has also been observed in \cite{condat_1D2D}
and transfers to our simplified relaxed real model.
It seems that this drawback only occurs for circle-valued images.
The benefit of ADMM and our simplified relaxed real model
with respect to the computation time 
are clearly visible.


\subsection{Denoising Colour Values}
\label{sec:denoise-col}

Sphere-valued data naturally occur as colour information of colour images.
For instance,
the hue of the HSV (hue, saturation, value) colour model is $\sphere_1$-valued,
and the chromaticity  of the CB (chromaticity, brightness) colour model
is $\sphere_2$-valued, 
see \cite[§~1.1]{CH2001}.
Identifying the circle $\sphere_1 = \{(\cos \alpha, \sin \alpha)^\tT : \alpha \in [0, 2\pi)\}$ with the angle $\alpha \in [0, 2\pi)$,
the hue starts with red at $0$, 
goes over green at $2\pi / 3$ and blue at $4\pi/3$,
and return to red at $2\pi$.
The chromaticity is the normalized colour vector of the RGB (red, green, blue) model. 
For the numerical simulations,
we take the hue and the chromaticity of the colour image in Figure~\ref{fig:hue_corals},
perturb the colour values according to the von Mises--Fisher distribution,
and apply Algorithm~\ref{alg:1} based on the simplified real model 
for denoising. 
The results are shown in Figure~\ref{fig:hue_corals} and \ref{fig:chroma_corals}.
Visually, most of the noise is removed.

\begin{figure}
\begin{minipage}[c]{0.3\textwidth}
\includegraphics[width=1.1\linewidth, clip=true, trim=10pt 0pt 60pt 0pt]{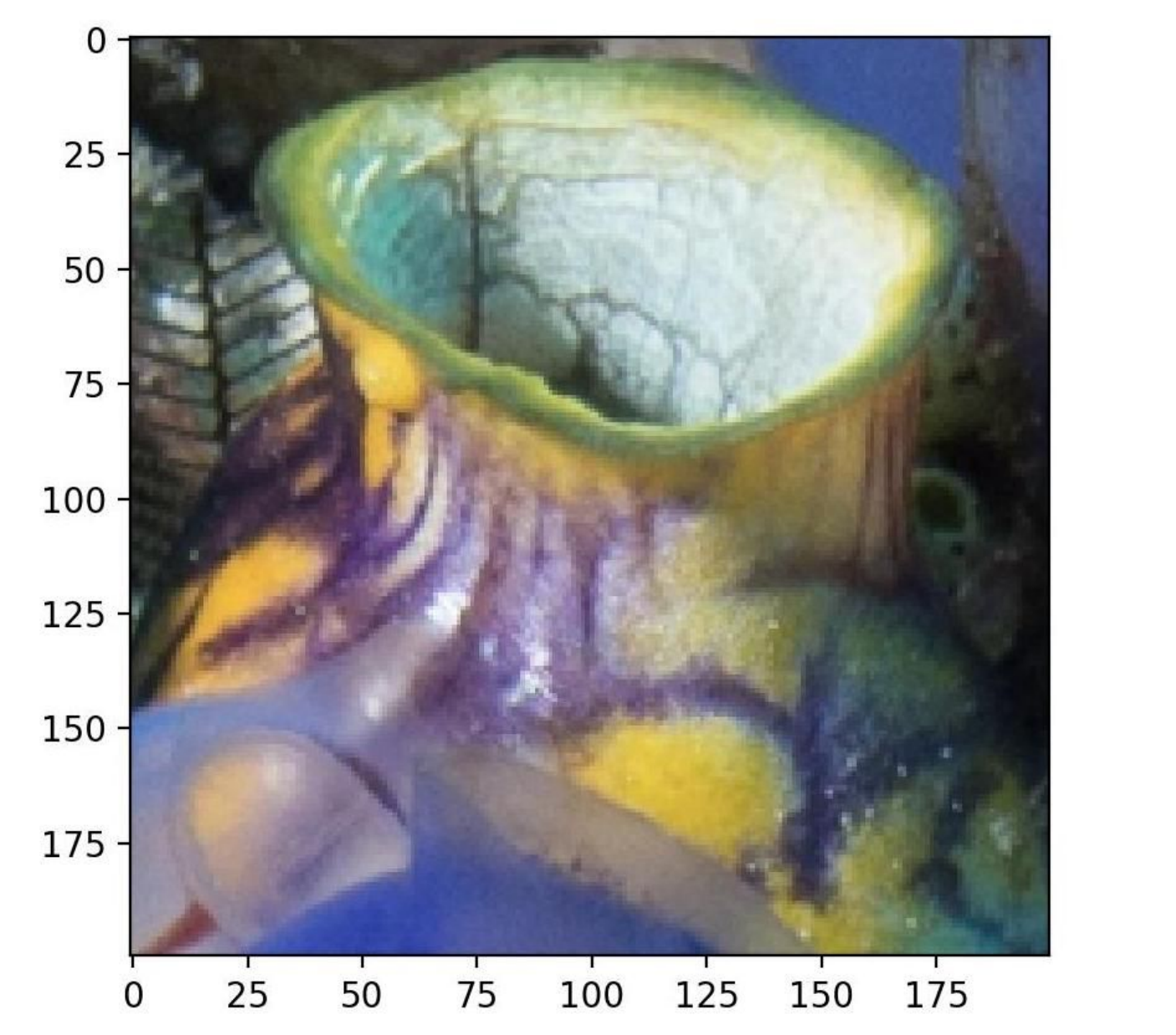}
\end{minipage}
\hfill
\begin{minipage}[c]{0.6\textwidth}
\caption{
\textit{Left:} Ground truth of size $200 \times 200$
for the colour denoising in Section~\ref{sec:denoise-col}. \\
\textit{Below:} Hue ground truth (left), noisy observation with $\kappa = 10$ (middle),
and denoised version (right)
of the hue experiment in Section~\ref{sec:denoise-col}
with respect to Figure~\ref{fig:hue_corals}.
ADMM has been employed with $w_n \coloneqq 1$,
$\lambda_{(n,m)} \coloneqq 1$ and $\rho \coloneqq 3$.
The solution comes
with an mean distance to the sphere $5.874\cdot10^{-4}$.}
\label{fig:hue_corals}
\end{minipage}
\hfill
\begin{minipage}[t]{\textwidth}
\includegraphics[width=\linewidth, clip=true, trim=135pt 0pt 135pt 0pt]{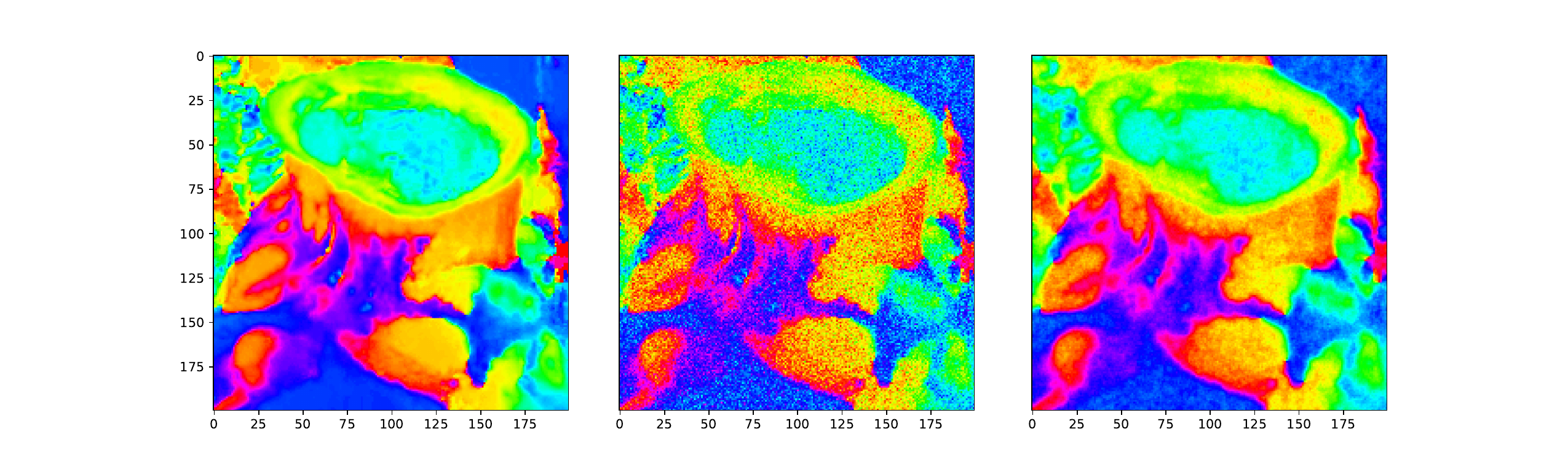}
\end{minipage}
\end{figure}

\begin{figure}
\includegraphics[width=\linewidth, clip=true, trim=135pt 0pt 135pt 0pt]{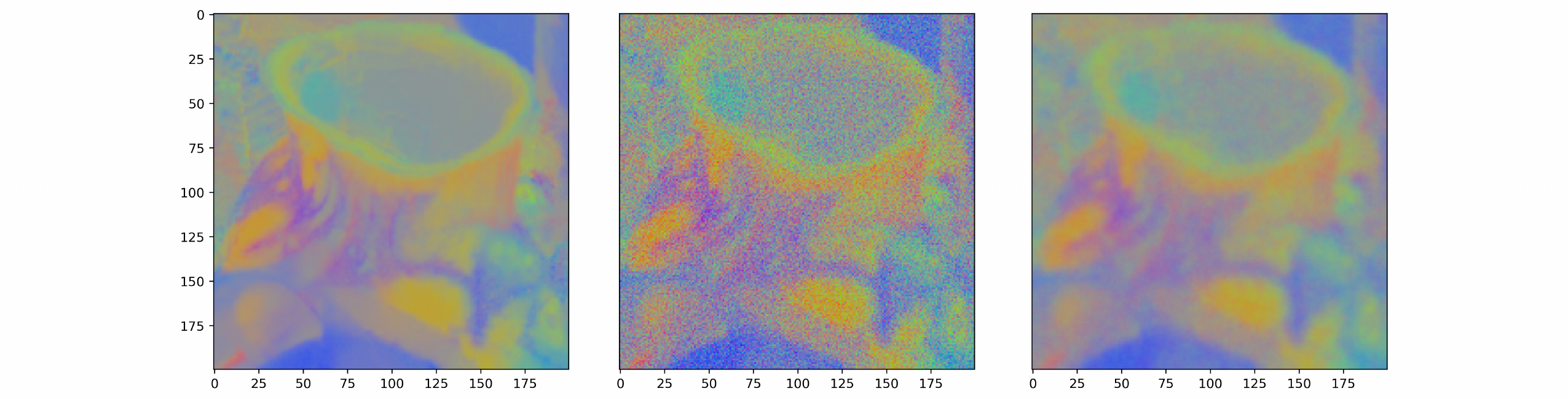}
\caption{
Ground truth (left), noisy observation with $\kappa = 100$ (middle),
and denoised version (right)
of the chromaticity experiment in Section~\ref{sec:denoise-col}
with respect to Figure~\ref{fig:hue_corals}.
ADMM has been employed with $w_n \coloneqq 1$,
$\lambda_{(n,m)} \coloneqq 3$,
and $\rho \coloneqq 3$.
The solution comes
with an mean distance to the sphere $2.343\cdot10^{-10}$.}
\label{fig:chroma_corals}
\end{figure}

\begin{SCfigure}[2]
\includegraphics[width=0.47\linewidth, clip=true, trim=30pt 7pt 30pt 2pt]{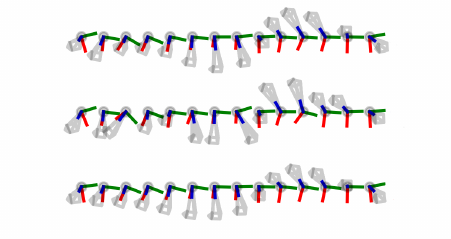}
\caption{
Ground truth (top), 
noisy observation with $\kappa_1 \coloneqq 30$ and $\kappa_2 \coloneqq 15$ (middle),
and denoised version (bottom)
of the $\SO(3)$-valued signal on the line graph in Section~\ref{sec:so3-data}.
The rotations are visualized by operating on a 3d object.
The red, green, and blue lines indicate the rotated unit vectors.
The entire signal has a length of $N \coloneqq 1000$,
where only the rotations at $n = 1, 11, 21, 31, \dots$ are visualized.
The parameters have been
$w_n \coloneqq 1$, 
$\lambda_{(n,m)} \coloneqq 50$, 
and $\rho \coloneqq 3$.
The solution comes
with an mean distance to the unit quaternions $3.245\cdot10^{-12}$.}
\label{fig:func_value_S3_1D}
\end{SCfigure}

\subsection{$\SO(3)$-Valued Data}       \label{sec:so3-data}

As discussed in Section~\ref{sec:sdc}, 
we want to employ Algorithm~\ref{alg:1} to denoise
3d rotation data on line and image graphs.
Starting from a smooth $\SO(3)$-valued signal $(\MR_n)_{n \in V}$
corresponding to the rotation axes $\Vv_n \in \sphere_2$ 
and rotation angles $\alpha_n \in \sphere_1$,
i.e.\ $\MR_n = \MR(\Vv_n, \alpha_n)$,
we generate disturbed rotation axes and angles
with respect to the von Mises--Fisher distributions on $\sphere_2$ 
and $\sphere_1$.
More precisely, 
we sample
\begin{linenomath*}
\begin{equation*}
    \Vw_n \sim \mathcal N_{\text{vMF}} (\Vv_n, \kappa_1)
    \quad\text{and}\quad
    \beta_n \sim \mathcal N_{\text{vMF}} (\alpha_n, \kappa_2).
\end{equation*}
\end{linenomath*}
The synthetic observations are then given by 
$(\MS_n)_{n\in V}$ with $\MS_n \coloneqq \MR(\Vw_n, \beta_n)$.
The aim is now to denoise the signal $(\MS_n)_{n \in V}$.
For this, 
we use the double cover $\SO(3) \cong \Hsphere / \{-1,1\}$
to represent the $3 \times 3$ matrices $\MS_n$
by their quaternion $q_n \coloneqq \pm q(\Vw_n,\beta_n)$, 
see \eqref{eq:quad-mat}.
Starting with $q_1 \coloneqq q(\Vw_1,\beta_1)$
the remaining signs of $q_n$ are here chosen such that
$\Re[q_n \bar q_m] \ge 0$ for all $(n,m) \in E$.
In this example,
the original signal is smooth
and the noise small enough
such that the sign choice is well defined.
Finally,
Algorithm~\ref{alg:1} is employed on the vector representations 
$\Vq_n \coloneqq (\Re[q_n], \Im_\I[q_n], \Im_\J[q_n], \Im_\K[q_n])^\tT$.
For the line graph,
see Figure~\ref{fig:func_value_S3_1D},
ADMM takes 8.60 seconds (209 iterations) to converge.
For the image graph, 
see Figure~\ref{fig:func_value_S3_2D},
ADMM needs 41.69 seconds (219 iterations) until convergence.  
In both cases,
the mean of $1- \lVert \Vq_n \rVert$ over $n \in V$
reaches $10^{-9}$ after around 200 iterations.
Hence,
the computed numerical solutions are in fact solutions
of the original unrelaxed problem.

\begin{figure*}[t]
\includegraphics[width=\linewidth, clip=true, trim=50pt 50pt 80pt 20pt]{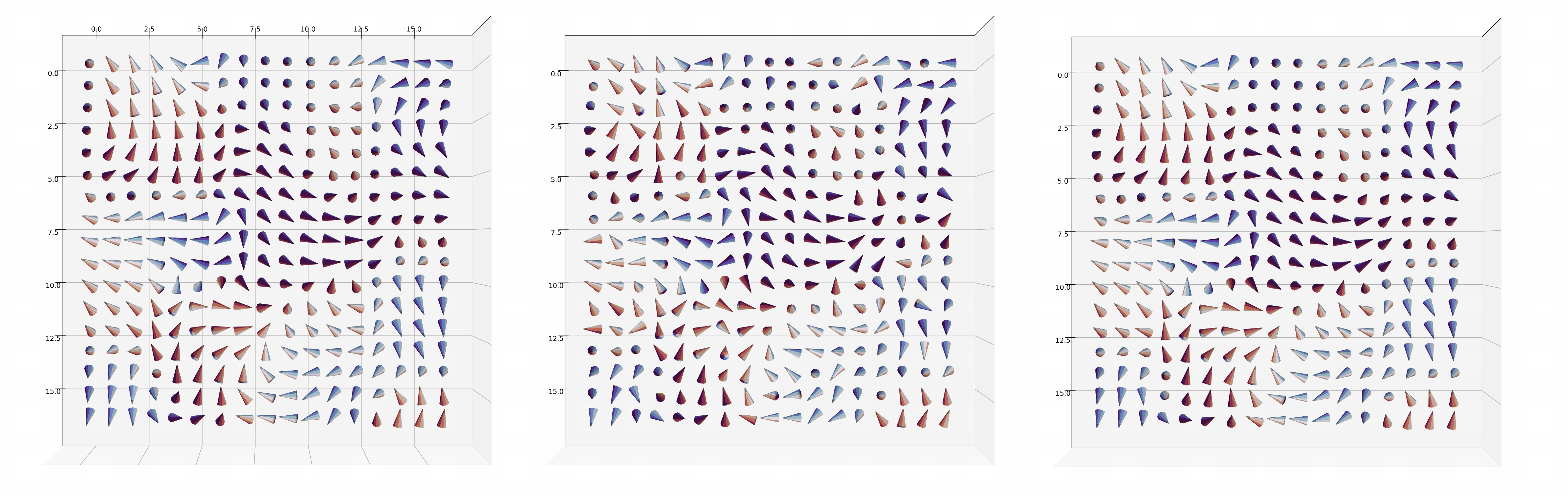}
\caption{
Ground truth (left), 
noisy observation with $\kappa_1 \coloneqq 30$ and $\kappa_2 \coloneqq 5$ (middle),
and denoised version (right)
of the $\SO(3)$-valued signal on the image graph in Section~\ref{sec:so3-data}.
The rotations are visualized by operating on a colored 3d cone.
The entire signal consists of $90 \times 90$ pixel,
where only every third pixel from the 25th to the 75th pixel is visualized.
The parameters have been
$w_n \coloneqq 1$, 
$\lambda_{(n,m)} \coloneqq 1$, 
and $\rho \coloneqq 3$.
The solution comes
with an mean distance to the unit quaternions $1.667\cdot10^{-10}$.}
\label{fig:func_value_S3_2D}
\end{figure*}

\section*{Appendix} 
\paragraph{Proof of Theorem~\ref{the:equiv-quad}.}
Recall the identification of the quaternion $z \in \HH$ with 
$\Vz \coloneqq (\Re[z], \Im_\I [z], \Im_\J[z], \Im_\K[z]) \in \R^4$
such that $\lvert z \rvert = \lVert \Vz \rVert$.
Moreover,
each quaternion can be identified with the matrix
\begin{linenomath*}
\begin{equation*}
    \Mquater(z) 
    \coloneqq
    \left[\begin{smallmatrix}
        \Re[z] & -\Im_\I[z] & -\Im_\J[z] & -\Im_\K[z] \\ 
        \Im_\I[z] & \Re[z] & -\Im_\K[z] & \Im_\J[z] \\
        \Im_\J[z] & \Im_\K[z] & \Re[z] & -\Im_\I[z] \\
        \Im_\K[z] & -\Im_\J[z] & \Im_\I[z] & \Re [z]
      \end{smallmatrix}
    \right]
    \in \R^{4 \times 4}.
\end{equation*}
\end{linenomath*}
In this way,
the addition and multiplication of quaternions carries over
to the addition and multiplication of the associated matrices.
The matrix $P_{(n,m)} \in \HH^{3 \times 3}$ in \eqref{eq:P-quad} is positive semi-definite
if and only if the real-valued representation
\begin{linenomath*}
\begin{equation}    
    \label{eq:mat_r4}
  \MP_{(n,m)} 
  \coloneqq
  \left[
    \begin{smallmatrix}
      \MI_4 & \Mquater(x_n) & \Mquater(x_m)\\
      \Mquater(x_n)^\tT & \MI_4 &   \Mquater(r_{(n,m)})^\tT\\
      \Mquater(x_m)^\tT & \Mquater(r_{(n,m)}) &  \MI_4
    \end{smallmatrix}
  \right]
  \in \R^{12 \times 12}
\end{equation}
\end{linenomath*}
is positive semi-definite.
In order to establish the equality
between quaternion model \eqref{eq:conv-quad-tik}
and the real-valued model in \eqref{eq:conv-real-tik} for $d=4$,
we compare the Schur complements of a permuted version of $\MP_{(n,m)}$ and of $\MQ_{(n,m)}$.
First,
we compute the Schur complement of $\MQ_{m,n}$ in \eqref{eq:mat_real_d4}
with respect to $\MI_4$,
which is again given by 
\begin{linenomath*}
\begin{equation*} 
  \MQ_{(n,m)}/ \MI_4 
  =
  \left[
    \begin{smallmatrix}
      1 - \lvert \Vx_n \rvert^2 
      & \Vell_{(n,m)} - \langle \Vx_n , \Vx_m\rangle\\
      \Vell_{(n,m)} - \langle \Vx_n , \Vx_m\rangle
      & 1 - \lvert \Vx_m \rvert^2
    \end{smallmatrix}
  \right].
\end{equation*}
\end{linenomath*}
\begin{table}
  \centering
  \rotatebox{90}{
  \begin{subcaptionblock}{0.95\textheight}
    \centering
      $
        \left[
          \scalebox{0.8}{$
            \begin{array}{*{2}{cccc|}cccc}
              1 & 0 & 0 & 0
              & \Re[x_n] & -\Im_\I[x_n] & -\Im_\J[x_n] & -\Im_\K[x_n]
              & \Re[x_m] & -\Im_\I[x_m] & -\Im_\J[x_m] & -\Im_\K[x_m]
              \\
              0 & 1 & 0 & 0
              & \Im_\I[x_n] & \Re[x_n] & -\Im_\K[x_n] & \Im_\J[x_n]
              & \Im_\I[x_m] & \Re[x_m] & -\Im_\K[x_m] & \Im_\J[x_m]
              \\
              0 & 0 & 1 & 0
              & \Im_\J[x_n] & \Im_\K[x_n] & \Re[x_n] & -\Im_\I[x_n]
              & \Im_\J[x_m] & \Im_\K[x_m] & \Re[x_m] & -\Im_\I[x_m]
              \\
              0 & 0 & 0 & 1
              & \Im_\K[x_n] & -\Im_\J[x_n] & \Im_\I[x_n] & \Re [x_n] 
              & \Im_\K[x_m] & -\Im_\J[x_m] & \Im_\I[x_m] & \Re [x_m]
              \\ \hline
              \Re[x_n] & \Im_\I[x_n] & \Im_\J[x_n] & \Im_\K[x_n]
              & 1 & 0 & 0 & 0
              & \Re[r_{(n,m)}] & \Im_\I[r_{(n,m)}] & \Im_\J[r_{(n,m)}] & \Im_\K[r_{(n,m)}]
              \\
              -\Im_\I[x_n] & \Re[x_n] & \Im_\K[x_n] & -\Im_\J[x_n]
              & 0 & 1 & 0 & 0
              & -\Im_\I[r_{(n,m)}] & \Re[r_{(n,m)}] & \Im_\K[r_{(n,m)}] & -\Im_\J[r_{(n,m)}]
              \\
              -\Im_\J[x_n] & -\Im_\K[x_n] & \Re[x_n] & \Im_\I[x_n]
              & 0 & 0 & 1 & 0
              & -\Im_\J[r_{(n,m)}] & -\Im_\K[r_{(n,m)}] & \Re[r_{(n,m)}] & \Im_\I[r_{(n,m)}]
              \\
              -\Im_\K[x_n] & \Im_\J[x_n] & -\Im_\I[x_n] & \Re [x_n]
              & 0 & 0 & 0 & 1 
              & -\Im_\K[r_{(n,m)}] & \Im_\J[r_{(n,m)}] & -\Im_\I[r_{(n,m)}] & \Re [r_{(n,m)}]
              \\ \hline
              \Re[x_m] & \Im_\I[x_m] & \Im_\J[x_m] & \Im_\K[x_m]
              & \Re[r_{(n,m)}] & -\Im_\I[r_{(n,m)}] & -\Im_\J[r_{(n,m)}] & -\Im_\K[r_{(n,m)}]
              & 1 & 0 & 0 & 0 
              \\
              -\Im_\I[x_m] & \Re[x_m] & \Im_\K[x_m] & -\Im_\J[x_m]
              & \Im_\I[r_{(n,m)}] & \Re[r_{(n,m)}] & -\Im_\K[r_{(n,m)}] & \Im_\J[r_{(n,m)}]
              & 0 & 1 & 0 & 0 
              \\
              -\Im_\J[x_m] & -\Im_\K[x_m] & \Re[x_m] & \Im_\I[x_m]
              & \Im_\J[r_{(n,m)}] & \Im_\K[r_{(n,m)}] & \Re[r_{(n,m)}] & -\Im_\I[r_{(n,m)}]
              & 0 & 0 & 1 & 0 
              \\
              -\Im_\K[x_m] & \Im_\J[x_m] & -\Im_\I[x_m] & \Re [x_m] 
              & \Im_\K[r_{(n,m)}] & -\Im_\J[r_{(n,m)}] & \Im_\I[r_{(n,m)}] & \Re [r_{(n,m)}]
              & 0 & 0 & 0 & 1 
            \end{array}$}
        \right]
        $
    \caption{Real-valued representation $\MP_{(n,m)}$ of $P_{(n,m)}$.}
    \label{tab:bigmatrix:orig}
  \end{subcaptionblock}}
  \quad
  \rotatebox{90}{
  \begin{subcaptionblock}{0.95\textheight}
    \centering
    $ \left[
        \scalebox{0.8}{$
          \begin{array}{cccc|*{8}{c}}
            1 & 0 & 0 & 0
            & \Re[x_n] & \Re[x_m]
            & -\Im_\I[x_n] & -\Im_\I[x_m]
            & -\Im_\J[x_n] & -\Im_\J[x_m]
            & -\Im_\K[x_n] & -\Im_\K[x_m]
            \\
            0 & 1 & 0 & 0
            & \Im_\I[x_n] & \Im_\I[x_m]
            & \Re[x_n] & \Re[x_m]
            & -\Im_\K[x_n] & -\Im_\K[x_m]
            & \Im_\J[x_n] & \Im_\J[x_m]           
            \\
            0 & 0 & 1 & 0
            & \Im_\J[x_n] & \Im_\J[x_m]
            & \Im_\K[x_n] & \Im_\K[x_m]
            & \Re[x_n] & \Re[x_m] 
            & -\Im_\I[x_n] & -\Im_\I[x_m]           
            \\
            0 & 0 & 0 & 1
            & \Im_\K[x_n] & \Im_\K[x_m]
            & -\Im_\J[x_n] & -\Im_\J[x_m]
            & \Im_\I[x_n] & \Im_\I[x_m]
            & \Re [x_n] & \Re [x_m] 
            \\ \hline
            \Re[x_n] & \Im_\I[x_n] & \Im_\J[x_n] & \Im_\K[x_n]
            & 1 & \Re[r_{(n,m)}]
            & 0 & \Im_\I[r_{(n,m)}]
            & 0 & \Im_\J[r_{(n,m)}]
            & 0 & \Im_\K[r_{(n,m)}]           
            \\
            \Re[x_m] & \Im_\I[x_m] & \Im_\J[x_m] & \Im_\K[x_m]
            & \Re[r_{(n,m)}] & 1
            & -\Im_\I[r_{(n,m)}] & 0
            & -\Im_\J[r_{(n,m)}] & 0
            & -\Im_\K[r_{(n,m)}] & 0
            \\
            -\Im_\I[x_n] & \Re[x_n] & \Im_\K[x_n] & -\Im_\J[x_n]
            & 0 & -\Im_\I[r_{(n,m)}]
            & 1 & \Re[r_{(n,m)}]
            & 0 & \Im_\K[r_{(n,m)}]
            & 0 & -\Im_\J[r_{(n,m)}]           
            \\
            -\Im_\I[x_m] & \Re[x_m] & \Im_\K[x_m] & -\Im_\J[x_m]
            & \Im_\I[r_{(n,m)}] & 0
            & \Re[r_{(n,m)}] & 1
            & -\Im_\K[r_{(n,m)}] & 0
            & \Im_\J[r_{(n,m)}] & 0
            \\
            -\Im_\J[x_n] & -\Im_\K[x_n] & \Re[x_n] & \Im_\I[x_n]
            & 0 & -\Im_\J[r_{(n,m)}]
            & 0 & -\Im_\K[r_{(n,m)}]
            & 1 & \Re[r_{(n,m)}]
            & 0 & \Im_\I[r_{(n,m)}]           
            \\
            -\Im_\J[x_m] & -\Im_\K[x_m] & \Re[x_m] & \Im_\I[x_m]
            & \Im_\J[r_{(n,m)}] & 0
            & \Im_\K[r_{(n,m)}] & 0
            & \Re[r_{(n,m)}] & 1
            & -\Im_\I[r_{(n,m)}] & 0
            \\
            -\Im_\K[x_n] & \Im_\J[x_n] & -\Im_\I[x_n] & \Re [x_n]
            & 0 & -\Im_\K[r_{(n,m)}]
            & 0 & \Im_\J[r_{(n,m)}]
            & 0 & -\Im_\I[r_{(n,m)}]
            & 1 & \Re [r_{(n,m)}]
            \\
            -\Im_\K[x_m] & \Im_\J[x_m] & -\Im_\I[x_m] & \Re [x_m] 
            & \Im_\K[r_{(n,m)}] & 0
            & -\Im_\J[r_{(n,m)}] & 0
            & \Im_\I[r_{(n,m)}] & 0
            & \Re [r_{(n,m)}] & 1
          \end{array}$}
      \right]
      $
    \caption{Permuted matrix $\tilde \MP_{(n,m)}$ build from $\MP_{(n,m)}$.}\label{tab:bigmatrix:perm}
  \end{subcaptionblock}}
    \quad  
  \rotatebox{90}{
  \begin{subcaptionblock}{0.95\textheight}
    \centering
    $ \left[
        \scalebox{0.6}{$
          \begin{array}{*{3}{cc|}cc}
            1 - \lvert x_n \rvert^2& \Re[r_{(n,m)}] - \Re[\bar x_m x_n]
            & 0 & \Im_\I[r_{(n,m)}] - \Im_\I[\bar x_m x_n]
            & 0 & \Im_\J[r_{(n,m)}] - \Im_\J[\bar x_m x_n]
            & 0 & \Im_\K[r_{(n,m)}] - \Im_\K[\bar x_m x_n]          
            \\
            \Re[r_{(n,m)}]  - \Re[\bar x_m x_n] & 1 - \lvert x_m \rvert^2
            & -\Im_\I[r_{(n,m)}] + \Im_\I[\bar x_m x_n]& 0
            & -\Im_\J[r_{(n,m)}] + \Im_\J[\bar x_m x_n]& 0
            & -\Im_\K[r_{(n,m)}] + \Im_\K[\bar x_m x_n]& 0
            \\ \hline
            0 & -\Im_\I[r_{(n,m)}] + \Im_\I[\bar x_m x_n]
            & 1 - \lvert x_n \rvert^2 & \Re[r_{(n,m)}]  - \Re[\bar x_m x_n]
            & 0 & \Im_\K[r_{(n,m)}] - \Im_\K[\bar x_m x_n]
            & 0 & -\Im_\J[r_{(n,m)}] + \Im_\J[\bar x_m x_n]          
            \\
            \Im_\I[r_{(n,m)}] - \Im_\I[\bar x_m x_n]& 0
            & \Re[r_{(n,m)}] - \Re[\bar x_m x_n] & 1 - \lvert x_m \rvert^2
            & -\Im_\K[r_{(n,m)}] + \Im_\K[\bar x_m x_n]& 0
            & \Im_\J[r_{(n,m)}] - \Im_\J[\bar x_m x_n]& 0
            \\ \hline
            0 & -\Im_\J[r_{(n,m)}] + \Im_\J[\bar x_m x_n]
            & 0 & -\Im_\K[r_{(n,m)}] + \Im_\K[\bar x_m x_n]
            & 1 - \lvert x_n \rvert^2 & \Re[r_{(n,m)}]  - \Re[\bar x_m x_n]
            & 0 & \Im_\I[r_{(n,m)}] - \Im_\I[\bar x_m x_n]          
            \\
            \Im_\J[r_{(n,m)}] - \Im_\J[\bar x_m x_n]& 0
            & \Im_\K[r_{(n,m)}] - \Im_\K[\bar x_m x_n]& 0
            & \Re[r_{(n,m)}] - \Re[\bar x_m x_n] & 1 - \lvert x_m \rvert^2
            & -\Im_\I[r_{(n,m)}] + \Im_\I[\bar x_m x_n]& 0
            \\ \hline
            0 & -\Im_\K[r_{(n,m)}] + \Im_\K[\bar x_m x_n]
            & 0 & \Im_\J[r_{(n,m)}] - \Im_\J[\bar x_m x_n]
            & 0 & -\Im_\I[r_{(n,m)}] + \Im_\I[\bar x_m x_n]
            & 1 - \lvert x_n \rvert^2 & \Re [r_{(n,m)}]  - \Re[\bar x_m x_n]
            \\
            \Im_\K[r_{(n,m)}] - \Im_\K[\bar x_m x_n]& 0
            & -\Im_\J[r_{(n,m)}] + \Im_\J[\bar x_m x_n]& 0
            & \Im_\I[r_{(n,m)}] - \Im_\I[\bar x_m x_n]& 0
            & \Re [r_{(n,m)}]  - \Re[\bar x_m x_n] & 1 - \lvert x_m \rvert^2
          \end{array}$}
      \right]
    $
    \caption{Schur complement $\tilde\MP_{(n,m)}/\MI_4$
        containing $\MQ_{(n,m)}/\MI_4$ on the diagonal.}\label{tab:bigmatrix:schur}
  \end{subcaptionblock}}
  \caption{Reformulating $P_{(n,m)}$ by writing out its real-valued representation $\MP_{(n,m)}$,
  permuting the columns/rows with respect to the permutation $(1,2,3,4,5,9,6,10,7,11,8,12)$,
  and calculating the Schur complement with respect to the upper left identity matrix.
  }
  \label{tab:bigmatrix}
\end{table}%
Second,
we consider the matrix $\MP_{(n,m)}$
whose details are given in Table~\ref{tab:bigmatrix}a.
Permuting the columns and rows of $\MP_{(n,m)}$
with respect to the permutation $(1,2,3,4,5,9,6,10,7,11,8,12)$,
we obtain $\tilde \MP_{(n,m)}$ in Table~\ref{tab:bigmatrix}b.
Forming the Schur complement with respect to the upper left identity
results in the block matrix $\tilde \MP_{(n,m)} / \MI_4$ in Table~\ref{tab:bigmatrix}c.

Comparing the Schur complements $\MQ_{(n,m)} / \MI_4$ and $\tilde \MP_{(n,m)} / \MI_4$
and using $\Re[\bar x_m x_n] = \langle \Vx_m, \Vx_n \rangle$, 
we obtain the following:		
if $(\hat {x}, \hat {r})$ solves \eqref{eq:conv-quad-tik}, then 
$(\hat{\Vx}, \hat \Vell)$ with $\hat \Vell_{(n,m)} \coloneqq \Re[\hat r_{(n,m)})]$ is a feasible point of \eqref{eq:conv-real-tik}.
Moreover, 
it holds $\FK(\hat{\Vx}, \hat \Vell) = \FJ(\hat {x}, \hat {r})$.
Conversely,	if $(\tilde{\Vx}, \tilde{\Vell})$
solves 
\eqref{eq:conv-real-tik}, then
$(\tilde {x}, \tilde {r})$ with $\tilde {r}$ defined in the assertion
is a feasible point of \eqref{eq:conv-quad-tik}---%
all off-diagonal blocks of $\tilde \MP_{(n,m)} / \MI_4$ become zero---%
and again $\FK(\tilde{\Vx}, \tilde \Vell) = \FJ(\tilde {x}, \tilde {r})$.
Then, taking the minimizing property into account, we conclude  
\begin{linenomath*}
$$
\FK(\tilde{\Vx}, \tilde \Vell) \le  
\FK(\hat{\Vx}, \hat \Vell) 
= 
\FJ(\hat {x}, \hat {r}) \le \FJ(\tilde{x}, \tilde {r}) = \FK(\tilde{\Vx}, \tilde \Vell)
$$
\end{linenomath*}
which is only possible if all values coincides. This yields the assertion.	
\hspace{\fill} $\square$

 \section*{Acknowledgements} 
 Funding by  the DFG excellence cluster MATH+ and by the 
 BMBF project “VI-Screen” (13N15754) 
 are gratefully acknowledged.
 Many thanks to L. Condat for discussions on the topic.

\bibliographystyle{abbrv}
\bibliography{literatur}

\begin{thebibliography}{10}

\bibitem{ASWK1993}
B.~L. Adams, S.~I. Wright, and K.~Kunze.
\newblock Orientation imaging: the emergence of a new microscopy.
\newblock {\em Metall. Mater. Trans. A Phys. Metall. Mater. Sci.}, 24:819--831,
  1993.

\bibitem{BBSW16}
M.~Ba{\v c}{\'a}k, R.~Bergmann, G.~Steidl, and A.~Weinmann.
\newblock A second order non-smooth variational model for restoring
  manifold-valued images.
\newblock {\em SIAM J. Sci. Comput.}, 38(1):A567--A597, 2016.

\bibitem{BHJPSW10}
F.~Bachmann, R.~Hielscher, P.~E. Jupp, W.~Pantleon, H.~Schaeben, and E.~Wegert.
\newblock Inferential statistics of electron backscatter diffraction data from
  within individual crystalline grains.
\newblock {\em J. Appl. Crystallogr.}, 43:1338--1355, 2010.

\bibitem{BHS11}
F.~Bachmann, R.~Hielscher, and H.~Schaeben.
\newblock Grain detection from 2d and 3d {EBSD} data—specification of the
  {MTEX} algorithm.
\newblock {\em Ultramicroscopy}, 111(12):1720--1733, 2011.

\bibitem{BB2018}
T.~Batard and M.~Bertalmio.
\newblock A geometric model of brightness perception and its application to
  color image correction.
\newblock {\em J. Math. Imaging Vis.}, 60(6):849--881, 2018.

\bibitem{bauschke}
H.~H. Bauschke and P.~L. Combettes.
\newblock {\em Convex analysis and monotone operator theory in {H}ilbert
  spaces}.
\newblock CMS Books in Mathematics/Ouvrages de Math\'{e}matiques de la SMC.
  Springer, New York, 2011.

\bibitem{BeiBre2024}
R.~Beinert and J.~Bresch.
\newblock Denoising sphere-valued data by relaxed total variation
  regularization, 2024.

\bibitem{BerChaHiePerSte16}
R.~Bergmann, R.~H. Chan, R.~Hielscher, J.~Persch, and G.~Steidl.
\newblock Restoration of manifold-valued images by half-quadratic minimization.
\newblock {\em Inverse Probl. Imaging}, 10(2):281--304, 2016.

\bibitem{BerLauPerSte19}
R.~Bergmann, F.~Laus, J.~Persch, and G.~Steidl.
\newblock Recent advances in denoising of manifold-valued images.
\newblock In R.~Kimmel and X.-C. Tai, editors, {\em Handbook of Numerical
  Analysis}, volume~20, pages 553--578. Elsevier, 2019.

\bibitem{BerLauSteWei14}
R.~Bergmann, F.~Laus, G.~Steidl, and A.~Weinmann.
\newblock Second order differences of cyclic data and applications in
  variational denoising.
\newblock {\em SIAM J. Imaging Sci.}, 7(4):2916--2953, 2014.

\bibitem{BPC+10}
S.~Boyd, N.~Parikh, E.~Chu, B.~Peleato, and J.~Eckstein.
\newblock Distributed optimization and statistical learning via the alternating
  direction method of multipliers.
\newblock {\em Found. Trends Mach. Learn.}, 3(1):1--122, 2010.

\bibitem{Bre93}
G.~E. Bredon.
\newblock {\em Topology and Geometry}.
\newblock Number 139 in Graduate Texts in Mathematics. Springer, New York,
  1993.

\bibitem{BurSawSte17}
M.~Burger, A.~Sawatzky, and G.~Steidl.
\newblock First order algorithms in variational image processing.
\newblock In R.~Glowinski, S.~Osher, and W.~Yin, editors, {\em Splitting
  methods in communication, imaging, science, and engineering}, Sci. Comput.,
  pages 345--407. Springer, Cham, 2017.

\bibitem{BRFa2000}
R.~B\"urgmann, P.~A. Rosen, and E.~J. Fielding.
\newblock Synthetic aperture radar interferometry to measure earth's surface
  topography and its deformation.
\newblock {\em Annu. Rev. Earth Planet Sci.}, 28(1):169--209, 2000.

\bibitem{CH2001}
T.~Chan, S.~Kang, and J.~Shen.
\newblock Total variation denoising and enhancement of color images based on
  the {CB} and {HSV} color models.
\newblock {\em J. Vis. Commun. Image Represent.}, 12(4):422--435, 2001.

\bibitem{condat_1D2D}
L.~Condat.
\newblock Tikhonov regularization of circle-valued signals.
\newblock {\em IEEE Trans. Signal Process.}, 70:2775--2782, 2022.

\bibitem{condat_3D}
L.~Condat.
\newblock Tikhonov regularization of sphere-valued signals, 2022.
\newblock arxiv:2207.12330.

\bibitem{CS13}
D.~Cremers and E.~Strekalovskiy.
\newblock Total cyclic variation and generalizations.
\newblock {\em J. Math. Imaging Vis.}, 47(3):258--277, 2013.

\bibitem{CK2018}
J.~Cremers and I.~Klugkist.
\newblock One direction? {A} tutorial for circular data analysis using {R} with
  examples in cognitive psychology.
\newblock {\em Front. Psychol.}, 9(2040), 2018.

\bibitem{DDT2011}
C.-A. Deledalle, L.~Denis, and F.~Tupin.
\newblock {NL}-{I}n{SAR}: Nonlocal interferogram estimation.
\newblock {\em IEEE Trans. Geosci. Remote Sens.}, 49(4):1441--1452, 2011.

\bibitem{graef12}
M.~Gr\"af.
\newblock A unified approach to scattered data approximation on {$\mathbb
  S^{3}$} and $\operatorname{SO}(3)$.
\newblock {\em Adv. Comput. Math.}, 37:379--392, 2012.

\bibitem{GNHSLP2022}
M.~Gr\"af, S.~Neumayer, R.~Hielscher, G.~Steidl, M.~Liesegang, and T.~Beck.
\newblock An optical flow model in electron backscatter diffraction.
\newblock {\em SIAM J. Imaging Sci.}, 15(1):228--260, 2022.

\bibitem{GS14}
P.~Grohs and M.~Sprecher.
\newblock Total variation regularization on {R}iemannian manifolds by
  iteratively reweighted minimization.
\newblock {\em Inf. Inference}, 5(4):353--378, 2016.

\bibitem{HJ13}
R.~A. Horn and C.~R. Johnson.
\newblock {\em Matrix Analysis}.
\newblock Cambridge University Press, Cambridge, 2nd edition, 2012.

\bibitem{HU2009}
D.~Huynh.
\newblock Metrics for 3d rotations: Comparison and analysis.
\newblock {\em SIAM J. Imaging Sci.}, 35:155--164, 10 2009.

\bibitem{ZMWW19}
Z.~Jia, M.~K. Ng, and W.~Wang.
\newblock Color image restoration by saturation-value total variation.
\newblock {\em SIAM Journal on Imaging Sciences}, 12(2):972--1000, 2019.

\bibitem{KS2002}
R.~Kimmel and N.~Sochen.
\newblock Orientation diffusion or how to comb a porcupine.
\newblock {\em J. Vis. Commun. Image. Represent.}, 13(1):238--248, 2002.

\bibitem{LEHS2008}
T.~Lan, D.~Erdogmus, S.~J. Hayflick, and J.~U. Szumowski.
\newblock Phase unwrapping and background correction in {MRI}.
\newblock In {\em Proceedings MLSP '08'}, pages 239--243, 2008.
\newblock {IEEE} Workshop on Machine Learning for Signal Processing, 16-19
  October 2008, Cancun, Mexico.

\bibitem{LNPS2017}
F.~Laus, M.~Nikolova, J.~Persch, and G.~Steidl.
\newblock A nonlocal denoising algorithm for manifold-valued images using
  second order statistics.
\newblock {\em SIAM J. Imaging Sci.}, 10(1):416--448, 2017.

\bibitem{LSKC13}
J.~Lellmann, E.~Strekalovskiy, S.~Koetter, and D.~Cremers.
\newblock Total variation regularization for functions with values in a
  manifold.
\newblock In {\em Proceedings ICCV '13}, pages 2944--2951, 2013.
\newblock {IEEE} International Conference on Computer Vision, 1--8 December
  2013, Sydney, Australia.

\bibitem{NikSte14}
M.~Nikolova and G.~Steidl.
\newblock Fast hue and range preserving histogram specification: theory and new
  algorithms for color image enhancement.
\newblock {\em IEEE Trans. Image Process.}, 23(9):4087--4100, 2014.

\bibitem{PB13}
N.~Parikh and S.~P. Boyd.
\newblock Proximal algorithms.
\newblock {\em Found. Trends Optim.}, 1:127--239, 2013.

\bibitem{pennec2006}
X.~Pennec.
\newblock Intrinsic statistics on {R}iemannian manifolds: basic tools for
  geometric measurements.
\newblock {\em SIAM J. Imaging Sci.}, 25:127--154, 2006.

\bibitem{PPS2017}
J.~Persch, F.~Pierre, and G.~Steidl.
\newblock Exemplar-based face colorization using image morphing.
\newblock {\em J. Imaging}, 3(4):48, 2017.

\bibitem{QKL2010}
M.~H. Quang, S.~H. Kang, and T.~M. Le.
\newblock Image and video colorization using vector-valued reproducing kernel
  {H}ilbert spaces.
\newblock {\em J. Math. Imaging Vis.}, 37:49--65, 2010.

\bibitem{RHJL2000}
P.~A. Rosen, S.~Hensley, I.~R. Joughin, F.~K. Li, S.~N. Madsen, E.~Rodriguez,
  and R.~M. Goldstein.
\newblock Synthetic aperture radar interferometry.
\newblock {\em Proc. IEEE}, 88(3):333--382, 2000.

\bibitem{RBBTK12}
G.~Rosman, A.~M. Bronstein, M.~M. Bronstein, X.-C. Tai, and R.~Kimmel.
\newblock Group-valued regularization for analysis of articulated motion.
\newblock In {\em Computer Vision -- ECCV '12}, pages 52--62, 2012.
\newblock Workshops and Demonstrations, 7-13 October 2012, Florence, Italy.

\bibitem{RTKB14}
G.~Rosman, X.-C. Tai, R.~Kimmel, and A.~Bruckstein.
\newblock Augmented-lagrangian regularization of matrix-valued maps.
\newblock {\em Methods and Applications of Analysis}, 21:121--138, 2014.

\bibitem{SRL2005}
Y.~Sowa, A.~D. Rowe, M.~C. Leake, T.~Yakushi, M.~Homma, A.~Ishijima, and R.~M.
  Berry.
\newblock Direct observation of steps in rotation of the bacterial flagellar
  motor.
\newblock {\em Nature}, 437:916--919, 2005.

\bibitem{SC11}
E.~Strekalovskiy and D.~Cremers.
\newblock Total variation for cyclic structures: convex relaxation and
  efficient minimization.
\newblock In {\em Proceedings CVPR '11}, pages 1905--1911, 2011.
\newblock {IEEE} Conference on Computer Vision and Pattern Recognition, 20-25
  June 2011, Colorado Springs, USA.

\bibitem{VO2002}
L.~A. Vese and S.~J. Osher.
\newblock Numerical methods for p-harmonic flows and applications to image
  processing.
\newblock {\em SIAM J. Numer. Anal.}, 40(6):2085--2104, 2002.

\bibitem{WDS14}
A.~Weinmann, L.~Demaret, and M.~Storath.
\newblock Total variation regularization for manifold-valued data.
\newblock {\em SIAM J. Imaging Sci.}, 7(4):2226--2257, 2014.

\bibitem{Zha97}
F.~Zhang.
\newblock Quaternions and matrices of quaternions.
\newblock {\em Linear Algebra Appl.}, 251:21--57, 1997.

\end{thebibliography}

\end{document}